\documentclass[11pt,times,twoside]{article}

\usepackage{graphicx,latexsym,euscript,makeidx,color,bm}
\usepackage{amsmath,amsfonts,amssymb,amsthm,thmtools,mathtools,mathrsfs,enumerate}
\usepackage[colorlinks,linkcolor=blue,anchorcolor=green,citecolor=red]{hyperref}

\usepackage{geometry}
\def\5n{\negthinspace \negthinspace \negthinspace \negthinspace \negthinspace }
\def\4n{\negthinspace \negthinspace \negthinspace \negthinspace }
\def\3n{\negthinspace \negthinspace \negthinspace }
\def\2n{\negthinspace \negthinspace }
\def\1n{\negthinspace }
     
       \def\oaB{\overleftarrow{B}}

\def\dbE{\mathbb{E}}     
\def\dbF{\mathbb{F}} \def\sF{\mathscr{F}}    
\def\dbG{\mathbb{G}}         
\def\dbH{\mathbb{H}} \def\sH{\mathscr{H}}

\def\dbN{\mathbb{N}}     
     
\def\dbP{\mathbb{P}}     
\def\dbQ{\mathbb{Q}}     
\def\dbR{\mathbb{R}}

\def\Om{\Omega}

\def\ss{\smallskip}                
\def\ms{\medskip}                
               
\def\ds{\displaystyle}

\def\no{\noindent}        \def\q{\quad}                      
\def\ns{\noalign{\ss}}    \def\qq{\qquad}                    
    \def\hb{\hbox}                     
                   
         \def\rf{\eqref}                    
  \def\deq{\triangleq}               
            \def\({\Big (}
\def\les{\leqslant}                  \def\){\Big )}
\def\ges{\geqslant}       \def\esssup{\mathop{\rm esssup}}   \def\[{\Big[}
           \def\]{\Big]}
                   \def\cd{\cdot}

\def\nn{\nonumber}        \def\ts{\times}

\def\a{\alpha}           \def\g{\gamma}   \def\O{\Omega}   \def\o{\omega}
\def\b{\beta}            \def\d{\delta}        
               
\def\e{\varepsilon}     \def\l{\lambda}        
         \def\f{\varphi}  \def\i{\infty}   

\def\bde{\begin{definition}\label}    \def\ede{\end{definition}}
\def\be{\begin{equation}}
\def\bel{\begin{equation}\label}      \def\ee{\end{equation}}
\def\bt{\begin{theorem}\label}        \def\et{\end{theorem}}
\def\bc{\begin{corollary}\label}      \def\ec{\end{corollary}}
\def\bl{\begin{lemma}\label}          \def\el{\end{lemma}}
\def\bp{\begin{proposition}\label}    \def\ep{\end{proposition}}
\def\bas{\begin{assumption}\label}    \def\eas{\end{assumption}}
\def\br{\begin{remark}\label}         \def\er{\end{remark}}
\def\bex{\begin{example}\label}       \def\ex{\end{example}}
\def\ba{\begin{array}}                \def\ea{\end{array}}
\def\ben{\begin{enumerate}}           \def\een{\end{enumerate}}

\newtheorem{theorem}{Theorem}[section]
\newtheorem{definition}[theorem]{Definition}
\newtheorem{proposition}[theorem]{Proposition}
\newtheorem{corollary}[theorem]{Corollary}
\newtheorem{lemma}[theorem]{Lemma}
\newtheorem{remark}[theorem]{Remark}
\newtheorem{example}[theorem]{Example}

\newtheorem{assumption}{Assumption}

\makeatletter
   
   \@addtoreset{equation}{section}
\makeatother

\allowdisplaybreaks[4]

\begin{document}

\title{\Large \bf Backward doubly stochastic differential equations and SPDEs with quadratic growth\thanks{Ying Hu was partially supported by Lebesgue Center of Mathematics ``Investissements d'avenir" program-ANR-11-LABX-0020-01, by CAESARS-ANR-15-CE05-0024 and by MFG-ANR16-CE40-0015-01. Jiaqiang Wen was supported by National Natural Science Foundation of China (Grant No. 12101291), Natural Science Foundation of Guangdong Province of China (Grant No. 2214050003543), and SUSTech start-up fund (Grant No. Y01286233). Jie Xiong was supported by National Natural Science Foundation of China (Grants No. 61873325 and 11831010), and SUSTech start-up fund (Grant No. Y01286120).}}
%

%

\author{Ying Hu\thanks{Univ. Rennes, CNRS, IRMAR - UMR 6625, F-35000 Rennes, France ((Email: {\tt ying.hu@univ-rennes1.fr}).}~,~~
Jiaqiang Wen\thanks{Department of Mathematics, Southern University of Science and Technology, Shenzhen, Guangdong, 518055, China ((Email: {\tt wenjq@sustech.edu.cn}).}~,~~
Jie Xiong\thanks{Department of Mathematics and SUSTech International center for Mathematics, Southern University of Science and Technology, Shenzhen, Guangdong, 518055, China (Email: {\tt xiongj@sustech.edu.cn}).}}
\maketitle

\no\bf Abstract. \rm
In this paper, we initiate the study of backward doubly stochastic differential equations (BDSDEs, for short) with quadratic growth. The existence, comparison, and stability results for one-dimensional BDSDEs are proved when the generator $f(t,Y,Z)$ grows in $Z$ quadratically and the terminal value is bounded, by introducing some new ideas.
Moreover, in this framework, we use BDSDEs to give a probabilistic representation for the solutions of semilinear stochastic partial differential equations (SPDEs, for short) in Sobolev spaces, and use it to prove the existence and uniqueness of such SPDEs, thus extending the nonlinear Feynman-Kac formula.
%
%
%
%

\ms

\no\bf Key words: \rm backward doubly stochastic differential equation, stochastic partial differential equation, quadratic growth, Feynman-Kac formula, Sobolev solution.

\ms

\no\bf AMS subject classifications. \rm 60H10, 60H15, 60H30.

\section{Introduction}

The nonlinear Feynman-Kac formula was introduced by Pardoux--Peng \cite{Pardoux-Peng-92} after they proved the existence and uniqueness of nonlinear backward stochastic differential equations (BSDEs, for short) in \cite{Pardoux-Peng-90}. It provides a probabilistic representation for a large class of semilinear partial differential equations (PDEs, for short).
A few years later, in order to give a probabilistic representation for semilinear stochastic PDEs (SPDEs, for short), a new class of BSDEs called the backward doubly stochastic differential equations (BDSDEs, for short) were introduced by them too in \cite{Pardoux-Peng-94}.
In order to present the work more clearly, we describe the problem in detail.

\ms

Let $(\Omega, \mathscr{F}, \mathbb{P})$ be a complete probability space on which two standard independent Brownian motions $\{W_t ; 0 \leqslant t<\i\}$ and $\{B_t ; 0 \leqslant t<\i\}$, with values in $\mathbb{R}^{d}$ and $\mathbb{R}^{l}$, respectively, are defined. Let $T>0$ be a fixed terminal time and denote by $\mathcal{N}$ the class of $\dbP$-null sets of $\sF$, where
\begin{equation*}
 \sF_{t} \deq   \sF_{t}^{W} \vee \sF_{t,T}^{B},\q \forall t\in[0,T].
\end{equation*}
In the above, for any process $\{\eta_t;0\les t\les T\},$ $\sF_{s,t}^{\eta}=\sigma\{\eta_r-\eta_s;s\les r\les t\}\vee \mathcal{N}$ and $\sF_{t}^{\eta}=\sF_{0,t}^{\eta}$.
For each $(t,x)\in[0,T]\ts\dbR^n$,  let $\{X^{t,x}_{s}; t\les s\les T\}$ be
the solution of the following stochastic differential equations (SDEs, for short):
\begin{equation*}
  X^{t,x}_{s} = x + \int_t^s b(X^{t,x}_{r}) dr + \int_t^s \sigma(X^{t,x}_{r}) dW_{r}, \q~ t\les s\les T,
\end{equation*}
and consider the following backward doubly stochastic differential equations: for $t\les s\les T$,
\begin{align}\label{BDSDE0}
 Y^{t,x}_s= h(X^{t,x}_T)+\int_s^Tf(r,X^{t,x}_r,Y^{t,x}_r,Z^{t,x}_r)dr
+\int_s^Tg(r,X^{t,x}_r,Y^{t,x}_r,Z^{t,x}_r)d\oaB_r - \int_s^TZ^{t,x}_{r}dW_{r},
\end{align}
where the $dW$ integral is a forward It\^{o}'s integral and the $d\overleftarrow{B}$ integral is a backward one. The coefficient $h$ is called the {\it terminal value} and the coefficient $f$ is called the {\it generator}. The solution of \rf{BDSDE0} is the pair $(Y,Z)$ of $\sF$-measurable processes.
%
%
For convenience, hereafter, by a {\it quadratic} BDSDE, or BDSDE {\it with quadratic growth}, we mean that in \rf{BDSDE0}, the generator $f$ grows in $Z$ quadratically. Meanwhile, we call $h$ the {\it bounded terminal value} if it is bounded.

\ms

Pardoux--Peng \cite{Pardoux-Peng-94} introduced BDSDEs \rf{BDSDE0}, gave the well-posedness of solutions under globally Lipschitz coefficients, and used it to prove the existence and uniqueness of the following semilinear SPDEs:
\begin{equation}\label{SPDE-2}
\begin{aligned}
\ds u(s, x)=& h(x)+\int_{s}^{T}\left\{\mathcal{L} u(r, x)+f\big(r, x, u(r, x),(\sigma^\top \nabla u)(r, x)\big)\right\} dr \\
\ns\ds &+\int_{s}^{T} g\big(r, x, u(r, x),(\sigma^\top \nabla u)(r, x)\big)d\oaB_r, \quad t \leqslant s \leqslant T,
\end{aligned}
\end{equation}
where $\sigma^\top$ denotes the transposed of $\sigma$, and
\begin{align*}
\mathcal{L}=\sum_{i=1}^{n} b_{i} \frac{\partial}{\partial x_{i}}+\frac{1}{2} \sum_{i, j=1}^{n} a_{i j} \frac{\partial^{2}}{\partial x_{i} \partial x_{j}}, \quad~\left(a_{i j}\right)=\sigma \sigma^\top.
\end{align*}
This result is summarized as the nonlinear stochastic Feynman-Kac formula,
which permits us to solve SPDE by BDSDE and, conversely, one can use SPDE to solve BDSDE too.
Since then, this important theory has attracted a lot of attention.
For example, Buckdahn--Ma \cite{Buckdahn01,Buckdahn012} introduced the stochastic viscosity solution to the nonlinear SPDEs, which connect BDSDEs to extend the nonlinear stochastic Feynman-Kac formula. Bally--Matoussi \cite{Bally-Matoussi-01} obtained the connection between the solution of BDSDEs and the Sobolev solution of related SPDEs. Zhang--Zhao \cite{Zhang-Zhao-07} studied infinite horizon BDSDEs and got the stationary solutions of related SPDEs.
Xiong \cite{Xio13} solved a long standing open problem about the SPDE characterization of the super-Brownian motion by making use of the BDSDE.
Buckdahn--Li--Xing \cite{Buckdahn-Li-Xing-21} studied mean-field BDSDEs and the associated nonlocal semilinear SPDEs.
Besides, Shi--Gu--Liu \cite{Shi-Gu-Liu-05} obtained the comparison theorem of solutions of BDSDEs. Han--Peng--Wu \cite{Han-Peng-Wu-10}  got the maximum principle for backward doubly stochastic control systems. Hu--Matoussi--Zhang \cite{Hu-Matoussi-Zhang-15} studied the Wong-Zakai approximations of BDSDEs.
For more recent developments of BDSDEs, we refer the readers to Boufoussi--Van Casteren--Mrhardy \cite{Boufoussi-Van Casteren-Mrhardy},
Gomez et al. \cite{Gomez-Xiong 2013}, Shi--Wen--Xiong \cite{Shi-Wen-Xiong-20}, Wu--Zhang \cite{Wu-Zhang-11}, Zhang--Zhao \cite{Zhang-Zhao-10}, etc.

\ms

Moreover, when the coefficient $g$ is absent, BDSDEs \rf{BDSDE0} are essentially reduced to the following backward stochastic differential equations:
\begin{align}\label{BSDE}
 Y_t=\xi+\int_t^Tf(s,Y_s,Z_s)ds - \int_t^TZ_{s}dW_{s},\q~0\les t\les T,
\end{align}
which were firstly introduced by Pardoux--Peng \cite{Pardoux-Peng-90}, where they obtained the existence and uniqueness under Lipschitz coefficients. Since then, BSDE stimulates some significant developments in many fields, such as partial differential equation (see Pardoux--Peng \cite{Pardoux-Peng-92} and Yong \cite{Yong-10}), mathematical finance (see El Karoui--Peng--Quenez \cite{Karoui-Peng-Quenez-97}), and stochastic optimal control (see Yong--Zhou \cite{Yong-Zhou-99}), to mention a few.
Meanwhile, based on the wide applications and motivated by the open problem proposed by Peng \cite{Peng-98}, many efforts have been made to relax the conditions on the generator $f$ of BSDE \rf{BSDE} for the existence and/or uniqueness of adapted solutions.
For example, Lepeltier--San Martin \cite{Lepeltier-Martin-97} proved the existence of adapted solutions for BSDE when the generator $f$ is continuous and of linear growth in $(Y,Z)$.
In 2000, Kobylanski \cite{Kobylanski-00} established the well-posedness for one-dimensional BSDE \rf{BSDE} with $f$ growing in $Z$ quadratically and with bounded terminal value.
When the terminal value is unbounded, Briand--Hu \cite{Briand-Hu-06,Briand-Hu-08} proved the existence and uniqueness of one-dimensional BSDEs with quadratic growth. Delbaen--Hu--Bao \cite{Delbaen-Hu-Bao11} studied BSDEs with the generator $f$ growing in $Z$ super-quadratically.
For multi-dimensional situation, Hu--Tang \cite{Hu-Tang-16} proved the existence and uniqueness of BDSDEs with diagonally quadratic generators, and Xing--Zitkovic \cite{Xing-Zitkovic} established the existence and uniqueness for a large class of Markovian BSDEs with quadratic growth.
Some other recent developments of quadratic BSDEs can be found in
Bahlali--Eddahbi--Ouknine \cite{Bahali-Eddahbi-Ouknine-17},
Barrieu--El Karoui \cite{Barrieu-Karoui-13}, Cheridito--Nam \cite{Cheridito-Nam-15},
Hu--Li--Wen \cite{Hu-Li-Wen-JDE2021}, Richou \cite{Richou-12}, Tevzadze \cite{Tevzadze-08},
 and references cited therein.

\ms

Note that in quadratic BSDEs, in order to overcome the difficulty of quadratic growth, one always would like to distinguish whether the terminal value is bounded or not. The reason is that when the terminal value is bounded, by involving bounded mean oscillation martingales (BMO martingale, for short), some nice estimates and regularities for the solution $(Y,Z)$ could be obtained. In other words, in this case, one could prove that $Y$ is bounded and $Z$ belongs to the BMO martingale space, which are useful when proving the existence and uniqueness.
In addition, when proving the existence, the main idea is to use the exponential transformation, which implies that there is an essential difference between the study of the one-dimensional case and the multi-dimensional case.
In summary, when studying the quadratic BSDEs, it is better to distinguish whether the terminal value is bounded or not, and the framework is a one-dimensional case or a multi-dimensional case.

\ms

On the other hand, during the past two decades, stimulated by the broad applications and the open problem of Peng \cite{Peng-98},
tremendous efforts have been made to relax the conditions on the generator $f$ of BDSDEs and extend the nonlinear stochastic Feynman-Kac formula.
%
%
%
However, to the best of our knowledge, there are few works concerning BDSDEs \rf{BDSDE0} when the generator $f$ grows in $Z$ quadratically. Some related studies are the works of Zhang--Zhao \cite{Zhang-Zhao-15} and Bahlali et al. \cite{Bahlali-17}, where they obtained the existence and uniqueness of BDSDEs when the generator $f$ is of polynomial growth in $Y$ and
is of super-linear (or sub-quadratic) in $Z$, respectively.
For quadratic BDSDEs, the main difficulty is that due to the collection $\dbF=\{ \sF_{t}; 0\les t\les T \}$ is neither increasing nor decreasing, and it does not constitute a filtration.
So the main technique of BMO martingale used in quadratic BSDEs is useless here, and one cannot expect to study the regularity of $Z$ in the BMO martingale space. Besides, the backward It\^o's integral appeared in BDSDE \rf{BDSDE0} will bring extra troubles when proving the existence and uniqueness of the solutions.

\ms

In this paper, we initiate the theory of BDSDEs with quadratic growth. In particular, we consider the one-dimensional framework of quadratic BDSDEs with bounded terminal value. First, we study the existence under a general assumption on the generator $f$. Borrowed some ideas from Kobylanski \cite{Kobylanski-00},
we construct an artful exponential transformation (see \autoref{Exp1}), which transforms the quadratic BDSDE into a classical one.
It should be pointed out that the solution  $(Y,Z)$ we considered is in the space
$L_{\dbF}^\i(0,T;\dbR)\times L_{\dbF}^2(0,T;\dbR^d)$ (see Section \ref{Sec2} for detailed definitions), since $\dbF=\{ \sF_{t}; 0\les t\les T \}$ is not a filtration.
Moreover, in order to overcome the troubles that come from the backward It\^o's integral, we introduce the idea of the comparison theorem of classical BDSDEs, and some restricted condition (see \rf{22.1.16.1}) on the coefficient $g$ is made, otherwise, $Y$ cannot be bounded (see \autoref{Eg1}).
Furthermore, based on a priori estimate (see \autoref{PriEstimate}) and the monotone stability  obtained (see \autoref{Mono-stable}), we prove that BDSDE with quadratic growth admits a solution $(Y,Z)\in L_{\dbF}^\i(0,T;\dbR)\times L_{\dbF}^2(0,T;\dbR^d)$ (see \autoref{existence}).
Second, in order to prove the uniqueness, we study a comparison theorem, due to the fact that the one-dimensional framework allows us to provide a comparison theorem which directly implies the uniqueness as a by-product (see \autoref{21.12.24.8}).
However, a stronger assumption on the coefficients $f$ and $g$ than that for the existence result is used to prove the comparison theorem. In other words, we assume that the partial derivatives of the coefficients $f$ and $g$ with respect to $Z$ are linear growth and bounded, respectively.
Moreover, based on the comparison theorem, we obtain a stability result, i.e.,
the solutions $(Y^{n}, Z^{n})$ of BDSDE with coefficients $(f^{n},g, \xi^{n})$ converge to the unique solution $(Y, Z)$ of BDSDE with coefficients $(f, g,\xi)$ when the sequences $\{f^{n};n\in\dbN\}$ and $\{\xi^{n};n\in\dbN\}$ converge to $f$ and $\xi$, respectively (see \autoref{Stability}).
Finally, we use BDSDE \rf{BDSDE0} with quadratic growth to give a probabilistic representation for the solutions of SPDE \rf{SPDE-2} in Sobolev spaces, and use it to prove the existence and uniqueness of SPDE \rf{SPDE-2}  when  $f$ grows in $\sigma^\top \nabla u$ quadratically  (see \autoref{22.3.9.1}), thus extending the nonlinear stochastic Feynman-Kac formula of Pardoux--Peng \cite{Pardoux-Peng-94} to quadratic growth situation.

\ms

The rest of the paper is organized as follows. In Section \ref{Sec2}, some preliminaries are presented. In Section \ref{Sec3}, a priori estimate, the monotonicity stability, and the existence of the solutions are proved. In Section \ref{Sec4}, we focus on the comparison theorem and thus derive the uniqueness of the solutions. The stability result is obtained in this section too. In Section \ref{Sec5}, we study the relationship between the solution of BDSDEs with quadratic growth and the Sobolev solution of related SPDEs. Section \ref{Sec6} concludes the results.

\section{Preliminaries}\label{Sec2}

Repeat that the triple $(\Omega, \mathscr{F}, \mathbb{P})$ is a complete probability space, and on which two standard independent Brownian motions $\{W_t ; 0 \leqslant t<\i\}$ and $\{B_t ; 0 \leqslant t<\i\}$ are defined with values in $\mathbb{R}^{d}$ and $\mathbb{R}^{l}$, respectively. Let $T>0$ be a fixed terminal time and denote by $\mathcal{N}$ the class of $\dbP$-null sets of $\sF$, where
\begin{equation*}
 \sF_{t} \deq   \sF_{t}^{W} \vee \sF_{t,T}^{B},\q~ \forall t\in[0,T].
\end{equation*}
In the above, for any process $\{\eta_t;0\les t\les T\},$ $\sF_{s,t}^{\eta}=\sigma\{\eta_r-\eta_s;s\les r\les t\}\vee \mathcal{N}$ and $\sF_{t}^{\eta}=\sF_{0,t}^{\eta}$.
Note that the collection $\dbF=\{ \sF_{t}; 0\les t\les T \}$ is not a filtration since it
is neither increasing nor decreasing.

\ms

Let us introduce some notations that will be used below. Denote by $|\cd|$ and $\langle\cd,\cd\rangle$ the Euclidean norm and dot product, respectively, throughout the paper.
For any $p\ges1$ and Euclidean space $\dbH$, denote by
$C^{p}(\dbH)$ the set of functions valued in $\dbH$ and of class $C^p$ with the partial
derivations of order less than or equal to $p$ are bounded.
In addition, we introduce the following spaces:
$$\ba{ll}
\ds L^p_{\sF_T}(\Om;\dbH)=\Big\{\xi:\Om\to\dbH\bigm|\xi\hb{ is $\sF_T$-measurable, }\|\xi\|_{L^p}\deq\big(\dbE|\xi|^p\big)^{1\over p}<\i\Big\},  \\
\ns\ds L_{\sF_T}^\i(\Om;\dbH)=\Big\{\xi:\Om\to\dbH\bigm|\xi\hb{ is $\sF_T$-measurable, }
\|\xi\|_{\i}\triangleq\esssup_{\o\in\Om}|\xi(\omega)|<\i\Big\},\ea$$
and for any $t\in[0,T)$ and $s\in[t,T]$, define
\begin{align*}
\ds L_\dbF^p(t,T;\dbH)\1n=&\ \1n\Big\{\f:[t,T]\1n\times\1n\Om\to\dbH\bigm|\f_s \hb{ is
$\sF_s$-measurable}, \
\|\f \|_{L_\dbF^p(t,T)}\deq\1n\(\dbE\int^T_t\1n|\f_s|^pds\)^{1\over p}\1n<\2n\i\Big\},\\
\ns\ds L_\dbF^\infty(t,T;\dbH)=&~
\Big\{\f:[t,T]\times\Om\to\dbH\bigm|\f_s \hb{ is $\sF_s$-measurable, }
\|\f \|_{L^\i_\dbF(t,T)}\deq \esssup_{(s,\o)\in[t,T]\times\Om}|\f_s(\o)|<\i\Big\},\\
\ns\ds S_\dbF^p(t,T;\dbH)=&~
\Big\{\f:[t,T]\times\Om\to\dbH\bigm|\f_s \hb{ is $\sF_s$-measurable, continuous, and }\\
\ns\ds &\qq\qq\qq\qq\q\ \|\f \|_{S_\dbF^p(t,T)}\deq\Big\{\dbE\[\sup_{s\in[t,T]}|\f_s|^p\]\Big\}^{\frac{1}{p}}<\i\Big\}.
%
\end{align*}
Consider the following backward doubly stochastic differential equations:
\begin{equation}\label{BDSDE}
Y_{t}=\xi+\int_{t}^{T} f\left(s,Y_s,Z_{s}\right) d s
+\int_{t}^{T} g\left(s,Y_s,Z_{s}\right) d \oaB_{s}
-\int_{t}^{T} Z_{s} d W_{s},\q~ 0\les t\les T,
\end{equation}
where $\xi:\Om\rightarrow\dbR^k$ is a random variable, $f:\Om\ts\dbR^{k}\ts\dbR^{k\ts d}\rightarrow\dbR^k$ and
$g:\Om\ts\dbR^{k}\ts\dbR^{k\ts d}\rightarrow\dbR^{k\ts l}$ are two coefficients.
Note that the $dW$ integral is a forward It\^{o}'s integral and the $d\overleftarrow{B}$ integral is a backward one, and both integrals are particular cases of It\^o-Skorohod integral (see Nualart \cite{Nualart-06}).

\begin{definition}\rm
A pair of measurable processes $(Y,Z)\in S^2_{\dbF}(0,T;\dbR^k)\times
L^2_\dbF(0,T;\dbR^{k\ts d})$ is called a solution of BDSDE \rf{BDSDE}, if $\dbP$-almost surely, it satisfies \rf{BDSDE}. In addition, if $Y$ is bounded, then we call the pair $(Y,Z)$ a bounded solution.
\end{definition}

Now, we recall the following propositions concerning BDSDEs, which come from Pardoux--Peng \cite{Pardoux-Peng-94} and Shi et al. \cite{Shi-Gu-Liu-05}. The first one is the It\^o's formula for backward doubly stochastic differential equations, the second one is the existence and uniqueness for the classical BDSDEs, and the third one is the comparison theorem for one-dimensional classical BDSDEs.

\begin{proposition}\label{Itoformula}\rm
Let $\Phi\in C^2(\dbR^k)$, $\alpha \in S^{2}_{\dbF}(0,T;\mathbb{R}^{k})$, $\beta\in L^{2}_{\dbF}(0,T;\mathbb{R}^{k})$, $\gamma\in L^{2}_{\dbF}(0,T;\mathbb{R}^{k\times l})$, $\delta\in L^{2}_{\dbF}(0,T;\mathbb{R}^{k\times d})$ be such that
\begin{align*}
  \a_t=\a_0 + \int_0^t \b_s ds + \int_0^t \g_s d\oaB_s + \int_0^t \d_s dW_s, \q~ 0\les t\les T.
\end{align*}
Then
\begin{align*}
  \Phi(\a_t)=&\ \Phi(\a_0)+\int_0^t \langle \Phi'(\a_s),\b_s\rangle ds
             +\int_0^t \langle \Phi'(\a_s),\g_s d\oaB_s\rangle
             + \int_0^t \langle \Phi'(\a_s),\d_sdW_s\rangle\\
           & -\frac{1}{2}\int_0^t \mathrm{Tr}[\Phi''(\a_s)\g_s\g_s^\top] ds
             +\frac{1}{2}\int_0^t \mathrm{Tr}[\Phi''(\a_s))\d_s\d_s^\top] ds.
\end{align*}
\end{proposition}

\begin{assumption}\label{Assump1}\rm
The terminal value $\xi$ comes from $L^2_{\sF_T}(\Omega;\dbR^k)$, and for any $y\in\dbR^k$ and  $z\in\dbR^{k\ts d}$,
$$f(\cd,y,z)\in L_\dbF^2(0,T;\dbR^k)\q~\hb{and}\q~ g(\cd,y,z)\in L_\dbF^2(0,T;\dbR^{k\ts l}).$$
Moreover, there exist two positive constants $C$ and $\a$ with $\a\in(0,1)$ such that for any $t\in[0,T]$, $y_1,y_2\in\dbR^k$, $z_1,z_2\in\dbR^{k\ts d}$,
\begin{align}
\ds |f(t,y_1,z_1)-f(t,y_2,z_2)|^2\les&\ C\big[|y_1-y_2|^2+|z_1-z_2|^2\big], \nn\\
\ns\ds  |g(t,y_1,z_1)-g(t,y_2,z_2)|^2\les&\ C|y_1-y_2|^2+\a|z_1-z_2|^2. \label{22.3.2.2}
\end{align}
\end{assumption}

\begin{proposition}\label{PP-EQ}\rm
Under \autoref{Assump1}, BDSDEs \rf{BDSDE} admit a unique solution  $(Y,Z)\in S^2_{\dbF}(0,T;\dbR^k)\times L^2_{\dbF}(0,T;\dbR^{k\ts d})$.
\end{proposition}

\begin{proposition}\label{Shi05}\rm
Suppose that the coefficients $(f^1,g,\xi^1)$ and $(f^2,g,\xi^2)$ satisfy \autoref{Assump1},
and let $(Y^1,Z^1)$ and $(Y^2,Z^2)$ be the solutions of one-dimensional BDSDEs \rf{BDSDE}  with parameters $(f^1,g,\xi^1)$ and $(f^2,g,\xi^2)$, respectively.  If for any $(t, y, z) \in[0, T] \times \dbR \times \dbR^{d}$,
$$\xi^{1} \les \xi^{2}, \q~ f^{1}(t, y, z) \les f^{2}(t, y, z), \q~ \text {a.s.,}$$
then we have that for any $t\in[0,T]$,
$$Y^{1}_t \les Y^{2}_t, \q~\text{a.s.}$$
\end{proposition}

\section{Quadratic BDSDE}\label{Sec3}
In this section, we study the existence of one-dimensional BDSDEs with quadratic growth, i.e., $k=1$.
Before proving the existence, we study a priori estimate and the monotone stability, which are important for us to prove existence later.
In order to justify the assumptions that were used to prove the existence, we give two examples. It should be pointed out that the first example also implies the main idea to prove existence, which is the central technique throughout this section.
%

\begin{example}\label{Exp1}\rm
Consider the following backward doubly stochastic differential equation,
\begin{equation}\label{3.1.0}
  Y_{t} = \xi + C\int_t^T |Z_{s}|^{2} ds +\a\int_t^T Z_{s} d\overleftarrow{B}_s - \int_t^T Z_{s} dW_s,\q~ 0\les t\les T,
\end{equation}
where $C>0$ and $\a\in(-1,1)$ are two constants. Let $\b$ be a positive constant which will be given later. Applying the exponential change of variable to $y=\exp(\b Y)$ implies that
\begin{align*}
\ds  y_T=&~y_t-C\int_t^T\b y_s|Z_{s}|^{2}ds
-\a\int_t^T\b y_sZ_{s} d\overleftarrow{B}_s+\int_t^T\b y_sZ_sdW_s\\
\ns\ds &-\frac{\a^2}{2}\int_t^T\b^2y_s|Z_{s}|^{2}ds
     +\frac{1}{2}\int_t^T\b^2y_s|Z_{s}|^{2}ds\\
\ns\ds =&~y_t-\Big[C-\frac{1-\a^2}{2}\b\Big]\int_t^T\b y_s|Z_{s}|^{2}ds-\a\int_t^T\b y_sZ_{s} d\overleftarrow{B}_s+\int_t^T\b y_sZ_sdW_s.
\end{align*}
Let $\b=\frac{2C}{1-\a^2}$ and denote by $z=\b yZ$, then the above equation can be rewritten as
\begin{align}\label{3.3}
  y_t=\exp(\b \xi)+\int_t^T\a z_s d\overleftarrow{B}_s-\int_t^Tz_sdW_s,\q~ 0\les t\les T,
\end{align}
which essential is a linear backward doubly stochastic differential equation. Then the classical existence and uniqueness (see  \autoref{PP-EQ}) implies that when
\bel{3.4.6}\exp(\b\xi)\in L^2_{\sF_T}(\O;\dbR),\ee
there exists a unique solution $(y,z)\in S^2_{\dbF}(0,T;\dbR)\times L^2_{\dbF}(0,T;\dbR^d)$ for BDSDE \rf{3.3}.
Note that, $\b$ is positive constant, so taking $\xi\in L^\i_{\sF_T}(\Omega;\dbR)$ is a sufficient condition to insure that $\xi$ satisfies \rf{3.4.6}. In fact, if $\xi\in L_{\sF_T}^\i(\Omega;\dbR)$, then $\exp(\xi)\in L_{\sF_T}^\i(\Omega;\dbR)$ and \rf{3.4.6} holds,
thus Eq. \rf{3.3} admits a unique solution $(y,z)\in S^2_{\dbF}(0,T;\dbR)\times L^2_{\dbF}(0,T;\dbR^d)$. Moreover, for BDSDE \rf{3.3}, by the comparison theorem (see \autoref{Shi05}), one has
$$  \exp(-\b\|\xi\|_\i)\les y_t\les \exp(\b\|\xi\|_\i),\q~ \forall t\in[0,T],$$
which implies that $y_t$ is positive and bounded. Now, we can define $(Y,Z)$ by
$$Y_t=\frac{\ln(y_t)}{\b},\q~ Z_t=\frac{z_t}{\b y_t},\q~ \forall t\in [0,T].$$
Then, it is easy to check that the pair $(Y,Z)\in L_{\dbF}^\i(0,T;\dbR)\times L_{\dbF}^2(0,T;\dbR^d)$  defined above is a solution of BDSDE \rf{3.1.0}.
Finally, the uniqueness of \rf{3.1.0} in $L_{\dbF}^\i(0,T;\dbR)\times L_{\dbF}^2(0,T;\dbR^d)$ follows from the uniqueness of \rf{3.3} and the fact that the exponential change of variable is no longer formal.
\end{example}
\begin{example}\label{Eg1}\rm
Consider the following linear backward doubly stochastic differential equation,
\begin{equation}\label{22.3.2.1}
  Y_{t} = \xi +\int_t^T(H+ d_sY_s) d\overleftarrow{B}_s - \int_t^T Z_{s} dW_s,\q~ 0\les t\les T,
\end{equation}
where $H$ is a constant and $d$ is a bounded deterministic function.
From \autoref{PP-EQ}, BDSDE \rf{22.3.2.1} admits a unique solution $(Y,Z)$.
Moreover, Proposition 2.1 of Bally--Matoussi \cite{Bally-Matoussi-01} implies that the solution $Y$ is given explicitly by
\begin{equation*}
  Y_{t} =\dbE\left[\Psi(t,T)\xi+H\int_t^T d_s\Psi(t,s)ds\bigg|\sF_t\right],\q\ 0\les t\les T,
\end{equation*}
where
$$\Psi(t,s)=\exp\left(\int_t^sd_rd\overleftarrow{B}_r+\frac{1}{2}\int_t^s|d_r|^2dr \right),\q\  t\les s\les T,$$
which is not bounded even if the terminal value $\xi$ is bounded.
\end{example}

\begin{remark}\label{Remark1} \rm
On the one hand, \autoref{Exp1} tells us that why we require the boundedness of the terminal value $\xi$. On the other hand, \autoref{Eg1} implies that even in the classical situation with bounded terminal value $\xi$, the normal Lipschitz condition of the coefficient $g$ with respect to $y$ (see \rf{22.3.2.2}) could not guarantee that $Y$ is bounded.
\end{remark}

\subsection{A priori estimate}

In this subsection, we prove a priori estimate for the solution $(Y,Z)$ of quadratic BDSDEs. First, we introduce the assumption.

\begin{assumption}\label{AssumpPri}\rm
Suppose that the terminal value $\xi\in L^\i_{\sF_T}(\Omega;\dbR)$.
Let $a:[0,T]\rightarrow\dbR$ and $b:[0,T]\rightarrow\dbR^+$ be two functions, and $C$ and $\a$ be two positive constants with $\a\in(0,1)$. For all $(t,y,z)\in[0,T]\times\dbR\times\dbR^d$,
\begin{align*}
  f(t,y,z)=a_0(t,y,z)y+f_0(t,y,z)
\end{align*}
with
\begin{align*}
a_0(t,y,z)\les a_t,\q~ |f_0(t,y,z)|\les b_t+C|z|^2,\q~\hb{a.s.}
\end{align*}
and
\begin{align}\label{22.1.16.1}
|g(t,y,z)|^2\les \a|z|^2,\q~\hb{a.s.}
\end{align}
\end{assumption}

\begin{remark} \rm
We point out that the boundedness of the solution $Y$ requires us to make the condition \rf{22.1.16.1} concerning the coefficient $g$; if not, the solution $Y$ may not be bounded even if the terminal value $\xi$ is bounded (see \autoref{Eg1} and \autoref{Remark1}).
Here we give an example that satisfying \rf{22.1.16.1}:
\begin{equation*}
  g(t,y,z)=\sin(y)l(z)\q~ \hb{with}\q~ |l(z)|^2\les \a|z|^2,\q~ \forall (t,y,z)\in[0,T]\ts\dbR\ts\dbR^d.
\end{equation*}
\end{remark}

\begin{proposition}[A priori estimate]\label{PriEstimate}
Let $(Y,Z)\in L_{\dbF}^\i(0,T;\dbR)\times L_{\dbF}^2(0,T;\dbR^d)$ be a solution of BDSDE \rf{BDSDE} with parameters $(f,g,\xi)$, which satisfy \autoref{AssumpPri} with $a, b, C, \a$, such that
$$a^{+}=\max(a,0),\q~ b\in L^1(0,T;\dbR).$$
Then for every $t\in[0,T]$,
\begin{align}
\ds  Y_t\les&\ \Big[\sup_{\Omega}(Y_T)\Big]^{+}\cdot\exp\Big(\int_t^Ta_sds\Big)
  +\int_t^Tb_s\exp\Big(\int_t^sa_rdr\Big)ds,\q~ \hb{a.s.} \label{3.4.5}\\
\ns\ds \bigg(\hb{resp. }  Y_t\ges&\ \Big[\inf_{\Omega}(Y_T)\Big]^{-}\cdot\exp\Big(\int_t^Ta_sds\Big)
  -\int_t^Tb_s\exp\Big(\int_t^sa_rdr\Big)ds,\q~ \hb{a.s.}\bigg). \label{3.4.5.2}
\end{align}
Moreover, there exists a positive constant $K$ depending only on $\|Y\|_{\i}, \ \|a^{+}\|_{L^1}$, $C$ and $\a$ such that
\begin{equation}\label{22.2.15.1}
\dbE\int_0^T|Z_s|^2ds\les K.
\end{equation}
\end{proposition}

\begin{proof}
Consider the following linear ordinary differential equation,
$$\f_t=\Big[\sup_{\Omega}(Y_T)\Big]^{+}+\int_t^T(a_s\f_s+b_s)ds,\q~ 0\les t\les T.$$
Then someone could compute that for $0\les t\les T$,
$$
  \f_t=\Big[\sup_{\Omega}(Y_T)\Big]^{+}\cdot\exp\Big(\int_t^Ta_sds\Big)
  +\int_t^Tb_s\exp\Big(\int_t^sa_rdr\Big)ds.
$$
So the result \rf{3.4.5} holds if we can prove that $Y_t\les\f_t$ for all $t\in[0,T]$. Applying It\^{o}'s formula (see \autoref{Itoformula}) to the process $Y_t-\f_t$ and to an increasing $C^2(\dbR)$ function $\Phi$, which will be determined later, deduce that
\begin{align*}
\ds  \Phi(Y_t-\f_t)=&\ \Phi(Y_T-\f_T)+\int_{t}^{T}\Phi'(Y_s-\f_s)[f(s,Y_s,Z_s)-(a_s\f_s+b_s)]ds\nonumber\\
\ns\ds  &\ +\int_{t}^{T}\Phi'(Y_s-\f_s)g(s,Y_s,Z_s)d\oaB_s
     -\int_{t}^{T}\Phi'(Y_s-\f_s)Z_sdW_s\nonumber\\
\ns\ds  &\ +\frac{1}{2}\int_{t}^{T}\Phi''(Y_s-\f_s)|g(s,Y_s,Z_s)|^2ds
     -\frac{1}{2}\int_{t}^{T}\Phi''(Y_s-\f_s)|Z_s|^2ds.\nn
\end{align*}
We denote that for every $s\in[t,T]$,
$$\tilde{a}_s=a_0(s,Y_s,Z_s).$$
From \autoref{AssumpPri}, note that the function $\Phi$ is increasing, we have that for every $s\in[t,T]$,
\begin{align*}
\ds \Phi'&(Y_s-\f_s)\big[f(s,Y_s,Z_s)-(a_s\f_s+b_s)\big]\\
\ns\ds  \les&\ \Phi'(Y_s-\f_s)\big[\tilde{a}_sY_s+b_s+C|Z_s|^2-(a_s\f_s+b_s)\big]\\
\ns\ds  \les&\ \Phi'(Y_s-\f_s)\big[\tilde{a}_s(Y_s-\f_s)+(\tilde{a}_s-a_s)\f_s+C|Z_s|^2\big],
\end{align*}
and
$$
\Phi''(Y_s-\f_s)|g(s,Y_s,Z_s)|^2\les\a\Phi''(Y_s-\f_s)|Z_s|^2.
$$
Due to that $(\tilde{a}_s-a_s)\f_s\les0$, we obtain
\begin{align}\label{3.4}
\ds  \Phi(Y_t-\f_t)\les&\ \Phi(Y_T-\f_T)+\int_{t}^{T}
      \tilde{a}_s\Phi'(Y_s-\f_s)(Y_s-\f_s)ds\nonumber\\
\ns\ds  &\ +\int_{t}^{T}\Big(C\Phi'-\frac{1-\a}{2}\Phi''\Big)(Y_s-\f_s)|Z_s|^2ds\\
\ns\ds  &\ +\int_{t}^{T}\Phi'(Y_s-\f_s)g(s,Y_s,Z_s)d\oaB_s
     -\int_{t}^{T}\Phi'(Y_s-\f_s)Z_sdW_s.\nonumber
\end{align}
Set $M=\|Y\|_\i+\|\f\|_\i$ and define the function $\Phi$ on the interval $[-M,M]$ by
$$\Phi(u)=\left\{\ba{ll}
\ds e^{\frac{2C}{1-\a}u}-1-\frac{2C}{1-\a}u-\frac{2C^2}{(1-\a)^2}u^2,\q\ & u\in[0,M],\\
\ns\ds0, & u\in[-M,0].
\ea\right.$$
Then it is easy to check that for all $u\in[-M,M]$,
\begin{align*}
\ds     \Phi(u)\ges0,\ \hb{ and }\ \Phi(u)=&\ 0\ \hb{ if and only if }\ u\les0, \nonumber\\
\ns\ds  \Phi'(u)\ges&\ 0,\nonumber\\
\ns\ds  0\les u\Phi'(u)\les&\ (M+1)\frac{2C}{1-\a}\Phi(u), \label{3.4.1}\\
\ns\ds  C\Phi'-\frac{1-\a}{2}\Phi''\les&\ 0. \nonumber
\end{align*}
Hence, if we set
$$k_t=a^{+}_t(M+1)\frac{2C}{1-\a},$$
then the function $k$ is positive and deterministic.
Note that $Y_T-\f_T\les0$ implies that $\Phi(Y_T-\f_T)=0$, and thus for all $0\les t\les T$, one has
\begin{equation}\label{22.3.2.3}
\begin{aligned}
\ds  0\les&\ \Phi(Y_t-\f_t)\les \int_{t}^{T}k_s\Phi(Y_s-\f_s)ds\\
\ns\ds  & +\int_{t}^{T}\Phi'(Y_s-\f_s)g(s,Y_s,Z_s)d\oaB_s
     -\int_{t}^{T}\Phi'(Y_s-\f_s)Z_sdW_s.
\end{aligned}
\end{equation}
Note that the term $\Phi'(Y_s-\f_s)$ is bounded on the interval $t\les s\les T$, so the terms $\Phi'(Y_s-\f_s)Z_s$ and $\Phi'(Y_s-\f_s)g(s,Y_s,Z_s)$ belong to the spaces
$L_{\dbF}^2(t,T;\dbR^d)$ and $L_{\dbF}^2(t,T;\dbR^l)$, respectively.
Taking expectations on both side of the inequality \rf{22.3.2.3} implies that
$$ 0\les \dbE\Phi(Y_t-\f_t)\les \dbE\int_{t}^{T}k_s\Phi(Y_s-\f_s)ds.$$
Then Gronwall's inequality yields that
$$\dbE\Phi(Y_t-\f_t)=0,\q~ \forall t\in[0,T].$$
Therefore, note that $\Phi(u)\ges0$, for all $t\in[0,T],$
$$\Phi(Y_t-\f_t)=0,\q~ \hb{ a.s.}$$
Finally, again, note that $\Phi(u)=0$ if and only if $u\les0$, we obtain that for every $t\in[0,T],$
$$Y_t-\f_t\les0,\q~\hb{ a.s.},$$
which implies that the inequality \rf{3.4.5} holds.
Similarly, by using the same computation, one can prove that \rf{3.4.5.2} holds too.

\ms

Finally, in order to prove the estimate \rf{22.2.15.1}, we use again \rf{3.4} with $t=0,\ \f=0$, $M=\|Y\|_\i$,
and define $\Phi$ on $[-M,M]$ by
\begin{equation*}\label{3.6.1}
\Phi(u)=\frac{1-\a}{2C^2}\Big[\exp\Big(\frac{2C}{1-\a}(u+M)\Big)-1-\frac{2C}{1-\a}(u+M)\Big].
\end{equation*}
It is easy to check that, for every $u\in[-M,M]$,
\begin{align*}
\ds  \Phi(u)\ges&\ 0,\\
\ns\ds  \Phi'(u)\ges&\ 0,\\
\ns\ds  0\les u\Phi'(u)\les&\ \frac{M}{C}\Big[\exp\Big(\frac{4CM}{1-\a}\Big)-1\Big],\\
\ns\ds  \frac{1-\a}{2}\Phi''-C\Phi'=&\ 1.
\end{align*}
Then \rf{3.4} implies that
\begin{align*}
  0\les\Phi(Y_0)\les&\ \Phi(Y_T)
  +\int_{0}^{T}a_s^{+}\frac{M}{C}\Big[\exp\Big(\frac{4CM}{1-\a}\Big)-1\Big]ds
  -\int_{0}^{T}|Z_s|^2ds\\
  &\ +\int_{0}^{T}\Phi'(Y_s)g(s,Y_s,Z_s)d\oaB_s
     -\int_{0}^{T}\Phi'(Y_s)Z_sdW_s,
\end{align*}
which leads that
\begin{align*}
  \dbE\int_{0}^{T}|Z_s|^2ds\les
  \Phi(M)+\frac{M}{C}\Big[\exp\Big(\frac{4CM}{1-\a}\Big)-1\Big]\|a^{+}\|_{L^1([0,T];\dbR)}.
\end{align*}
This completes the proof.
\end{proof}

An immediate consequence of this proposition is the following corollary.

\begin{corollary}
Let $(Y,Z)\in L_{\dbF}^\i(0,T;\dbR)\times L_{\dbF}^2(0,T;\dbR^d)$ be a solution of BDSDE with parameters $(f,g,\xi)$, where  $\xi$ belongs to $ L^\i_{\sF_T}(\Omega;\dbR)$, and $f$ and $g$ satisfy the condition \autoref{AssumpPri} with $a,b$, $C>0$ and $\a\in(0,1)$ such that $a^{+},b\in L^1(0,T;\dbR)$, then
$$\|Y\|_{\i}\les\big(\|\xi\|_{\i}+\|b\|_{L^1([0,T];\dbR)}\big)\exp\big(\|a^+\|_{L^1([0,T];\dbR)}\big).$$
\end{corollary}

\subsection{Monotone stability}
In this subsection, we would like to study the monotone stability of BDSDEs \rf{BDSDE} with quadratic growth, which is also important for us to prove the existence later.

\begin{proposition}[Monotone stability]\label{Mono-stable}
Let $(f,g,\xi)$ be a set of parameters and let $(f^n,g,\xi^n)$ be a sequence of parameters such that:
\begin{itemize}
\item [$\rm(i)$] There exist two positive constants $C$ and $\a$ with $0<\a<1$ such that for each $n\in\dbN$,
      \begin{align}
        \ds &|f^n(t,y,z)|\les C(1+|y|+|z|^2),\q~  |g(t,y,z)|^2\les\a|z|^2,\q~ \forall  t\in[0,T], y\in\dbR, z\in\dbR^d,\label{22.2.17.4}\\
        \ns\ds & |g(t,y_1,z_1)-g(t,y_2,z_2)|^2\les C|y_1-y_2|^2+\a|z_1-z_2|^2,\q~ \forall t\in[0,T], y_1,y_2\in\dbR, z_1,z_2\in\dbR^d.\nn
      \end{align}
  \item [$\rm(ii)$]  The sequence $\{f^n;n\in\dbN\}$ converges to $f$ locally uniformly on $[0,T] \times \mathbb{R} \times \mathbb{R}^{d}$, and for each $n \in \mathbb{N},$  the sequence $\{\xi^n;n\in\dbN\}$  converges to $\xi$ in $L^\i_{\sF_T}(\Omega;\dbR)$.
  \item [$\rm(iii)$] For each $n\in\dbN$, the BDSDE with parameters $(f^n,g,\xi^n)$ has a solution
  $$(Y^n,Z^n)\in L_{\dbF}^\i(0,T;\dbR)\times L_{\dbF}^2(0,T;\dbR^d),$$
  and the sequence $\{Y^n;n\in\dbN\}$ is monotonic such that for every $n\in\dbN$, $\|Y^n\|_\i\les M$, where $M$ is a positive constant.
\end{itemize}
Then there exists a pair of processes $(Y,Z)\in L_{\dbF}^\i(0,T;\dbR)\times L_{\dbF}^2(0,T;\dbR^d)$ such that
\begin{equation*}
\left\{\begin{aligned}
\ds  &\lim_{n\rightarrow\i}Y^n=Y \hb{ uniformy on } [0,T],\\
\ns\ds  &\{Z^n;n\in\dbN\} \hb{ converges to $Z$ in } L_{\dbF}^2(0,T;\dbR^d).
\end{aligned}\right.
\end{equation*}
Moreover, $(Y,Z)$ is a solution of BDSDE with parameters $(f,g,\xi)$. In particular, if the sequence $\{Y^n;n\in\dbN\}$ has continuous paths for every $n\in\dbN$, then $Y$ has continuous paths too.
\end{proposition}

\begin{proof}
From the assumptions we see that for all $t \in [0,T]$, the sequence $\{Y^n_t;n\in\dbN\}$ is monotonic and bounded, and therefore it has a limit, which denoted by $Y$.
Without lose of generality, we consider that the sequence $\{Y^n;n\in\dbN\}$ is monotonic increasing to $Y$.
Moreover, from \autoref{PriEstimate},  there is a  positive constant $K$ such that,
$$\dbE\int_0^T|Z^n|^2ds\les K,\q~ \forall n\in\dbN.$$
Then there exists a process $Z\in L_{\dbF}^2(0,T;\dbR^d)$ such that a subsequence of $\{Z^n;n\in\dbN\}$ converges to $Z$ weakly in $L_{\dbF}^2(0,T;\dbR^d)$.
For simplicity presentation, by otherwise choosing a subsequence, we may assume that the whole sequence $\{Z^n;n\in\dbN\}$  converges to $Z$ weakly in $L_{\dbF}^2(0,T;\dbR^d)$.

\ms

Next, we would like to divide the proof into two steps, and denote the following notations:
\begin{align*}
\ds  \Delta Y^n\deq Y-Y^n,\q\ \Delta Y^{m,n}\deq Y^m-Y^n,\\
\ns\ds  \Delta Z^n\deq Z-Z^n,\q\  \Delta Z^{m,n}\deq Z^m-Z^n.
\end{align*}

$\emph{Step 1.}$ The sequence  $\{Z^n;n\in\dbN\}$ strongly converges to $Z$ in  $L_{\dbF}^2(0,T;\dbR^d)$, i.e.,
\begin{equation*}\label{Zn}
  \lim_{n\rightarrow\i}\dbE\int_0^T|\Delta Z^n_t|^2dt=0.
\end{equation*}
The main arguments of this step come from Zhang \cite{Zhang-17} and Kobylanski \cite{Kobylanski-00}.
It is easy to see that the pair $(\Delta Y^{m,n},\Delta Z^{m,n})$ satisfies the following equation,
$$\left\{\ba{ll}
\ds -d\Delta Y^{m,n}_t= [f^m(t,Y^m_t,Z^m_t)-f^n(t,Y^n_t,Z^n_t)]dt
\\
\ns\ns\ds\qq\qq\q\ \ +[g(t,Y^m_t,Z^m_t)-g(t,Y^n_t,Z^n_t)]d\oaB_t
- \Delta Z^{m,n}_tdW_{t},\q~t\in [0,T],\\
\ns\ns\ds\q\ \Delta Y^{m,n}_T=\ \xi^m-\xi^n.
\ea\right.$$
By assumptions, there exists a constant $C_0$ depending on $C$ and $M$, such that for all $m,n\in\dbN$, $t\in[0,T]$,
\begin{equation}\label{3.6}
\begin{aligned}
\ds |f^m(t,Y^m_t,Z^m_t)-f^n(t,Y^n_t,Z^n_t)|\les&\ C_0[1+|\Delta Z^{m,n}_t|^2+|\Delta Z^n_t|^2+|Z_t|^2],\\
\ns\ds |g(t,Y^m_t,Z^m_t)-g(t,Y^n_t,Z^n_t)|^2\les&\ C|\Delta Y^{m,n}_t|^2+\a|\Delta Z^{m,n}_t|^2.
\end{aligned}
\end{equation}
%
%
Let $\Phi:\dbR^+\rightarrow\dbR^+$ be a smooth increasing function which will be specified later. For $m\ges n$, using It\^{o}'s formula (see \autoref{Itoformula}) to $\Phi(\Delta Y^{m,n})$ we have that
\begin{align*}
\ds d\Phi(\Delta Y^{m,n}_t)=& -\Phi'(\Delta Y^{m,n}_t)[f^m(t,Y^m_t,Z^m_t)-f^n(t,Y^n_t,Z^n_t)]dt\\
\ns\ds & -\Phi'(\Delta Y^{m,n}_t)[g(t,Y^m_t,Z^m_t)-g(t,Y^n_t,Z^n_t)]d\oaB_t+\Phi'(\Delta Y^{m,n}_t)\Delta Z^{m,n}_tdW_{t}\\
\ns\ds & -\frac{1}{2}\Phi''(\Delta Y^{m,n}_t)|g(t,Y^m_t,Z^m_t)-g(t,Y^n_t,Z^n_t)|^2dt
    +\frac{1}{2}\Phi''(\Delta Y^{m,n}_t)|\Delta Z^{m,n}_t|^2dt\\
\ns\ds \ges&-C_0\Phi'(\Delta Y^{m,n}_t)[1+|\Delta Z^{m,n}_t|^2+|\Delta Z^{n}_t|^2+|Z_t|^2]dt\\
\ns\ds & -\Phi'(\Delta Y^{m,n}_t)[g(t,Y^m_t,Z^m_t)-g(t,Y^n_t,Z^n_t)]d\oaB_t
+\Phi'(\Delta Y^{m,n}_t)\Delta Z^{m,n}_tdW_{t}\\
\ns\ds &-\frac{C}{2}\Phi''(\Delta Y^{m,n}_t)|\Delta Y^{m,n}_t|^2dt
 -\frac{\a}{2}\Phi''(\Delta Y^{m,n}_t)|\Delta Z^{m,n}_t|^2dt
    +\frac{1}{2}\Phi''(\Delta Y^{m,n}_t)|\Delta Z^{m,n}_t|^2dt,
\end{align*}
where we have use the inequality \rf{3.6}.

\ms

In the following, we expect $\Delta Z^{m,n}$ and $\Delta Z^{n}$ to be closed. For this purpose, define
$$\Phi(u)=\frac{1-\a}{8C_0^2}\Big[\exp(\frac{4C_0}{1-\a}u)-\frac{4C_0}{1-\a}u-1\Big].$$
It is straightforward to check that for $u\ges0$, we have
\begin{align*}
\ds  &\Phi(u)\ges 0\q\hb{with}\q\Phi(0)=0,\\
\ns\ds  &\Phi'(u)\ges 0\q\hb{with}\q\Phi'(0)=0,\\
\ns\ds & \frac{1-\a}{2}\Phi''(u)=4C_0\Phi'(u)+ 2.
\end{align*}
Note that  $Y^m\ges Y^n$ when $m\ges n$. Then
\begin{align*}
\ds  d\Phi(\Delta Y^{m,n}_t)\ges&
-C_0\Phi'(\Delta Y^{m,n}_t)[1+|\Delta Z^{m,n}_t|^2+|\Delta Z^{n}_t|^2+|Z_t|^2]dt\\
\ns\ds    & -\Phi'(\Delta Y^{m,n}_t)[g(t,Y^m_t,Z^m_t)-g(t,Y^n_t,Z^n_t)]d\oaB_t
+\Phi'(\Delta Y^{m,n}_t)\Delta Z^{m,n}_tdW_{t}\\
\ns\ds    & +[4C_0\Phi'(\Delta Y^{m,n}_t)+ 2]|\Delta Z^{m,n}_t|^2dt
-\frac{C}{2}\Phi''(\Delta Y^{m,n}_t)|\Delta Y^{m,n}_t|^2dt\\
\ns\ds    =&-C_0\Phi'(\Delta Y^{m,n}_t)[1+|\Delta Z^{n}_t|^2+|Z_t|^2]dt
-\frac{C}{2}\Phi''(\Delta Y^{m,n}_t)|\Delta Y^{m,n}_t|^2dt\\
\ns\ds    & -\Phi'(\Delta Y^{m,n}_t)[g(t,Y^m_t,Z^m_t)-g(t,Y^n_t,Z^n_t)]d\oaB_t
+\Phi'(\Delta Y^{m,n}_t)\Delta Z^{m,n}_tdW_{t}\\
\ns\ds    & +[3C_0\Phi'(\Delta Y^{m,n}_t)+ 2]|\Delta Z^{m,n}_t|^2dt,
\end{align*}
and thus
\begin{align*}
\ds &\dbE\Phi(\Delta Y^{m,n}_0)
+\dbE\int_0^T[3C_0\Phi'(\Delta Y^{m,n}_t)+ 2]|\Delta Z^{m,n}_t|^2dt
-C_0\dbE\int_0^T\Phi'(\Delta Y^{m,n}_t)|\Delta Z^{n}_t|^2dt \\
\ns\ds \les&\dbE\Big[\Phi(\Delta Y^{m,n}_T)
+C_0\int_0^T\Phi'(\Delta Y^{m,n}_t)[1+|Z_t|^2]dt
+\frac{C}{2}\int_t^T\Phi''(\Delta Y^{m,n}_t)|\Delta Y^{m,n}_t|^2dt\Big].
\end{align*}
Due to that $\{Z^m;m\ges1\}$ converges to $Z$ weakly in $L_{\dbF}^2(0,T;\dbR^d)$, so next we would like to
 fix $n$ and pass to the limit as $m\rightarrow\i$.
The convergence of $Y^m\rightarrow Y$ being pointwise, and $Y^m$ being bounded by $M$, by Lebesgue's dominated convergence theorem, we have that
\begin{equation*}
\begin{aligned}
\ds &\liminf_{m\rightarrow\i}\dbE\int_0^T[3C_0\Phi'(\Delta Y^{n}_t)+ 2]|\Delta Z^{m,n}_t|^2dt
-C_0\dbE\int_0^T\Phi'(\Delta Y^{n}_t)|\Delta Z^{n}_t|^2dt \\
\ns\ds \les&\dbE\Big[\Phi(\Delta Y^{n}_T)
+C_0\int_0^T\Phi'(\Delta Y^{n}_t)[1+|Z_t|^2]dt
+\frac{C}{2}\int_0^T\Phi''(\Delta Y^{n}_t)|\Delta Y^{n}_t|^2dt\Big],
\end{aligned}
\end{equation*}
and as
\begin{align*}
 \liminf_{m\rightarrow\i}\bigg[-\dbE\int_0^T\Phi'(\Delta Y^{n}_t)|\Delta Z^{m,n}_t|^2dt\bigg]
\les -\dbE\int_0^T\Phi'(\Delta Y^{n}_t)|\Delta Z^{n}_t|^2dt,
\end{align*}
we obtain that
\begin{equation}\label{22.2.17.3}
\begin{aligned}
\ds &\liminf_{m\rightarrow\i}\dbE\int_0^T[2C_0\Phi'(\Delta Y^{n}_t)+ 2]|\Delta Z^{m,n}_t|^2dt\\
\ns\ds \les&\dbE\Big[\Phi(\Delta Y^{n}_T)
+C_0\int_0^T\Phi'(\Delta Y^{n}_t)[1+|Z_t|^2]dt
+\frac{C}{2}\int_0^T\Phi''(\Delta Y^{n}_t)|\Delta Y^{n}_t|^2dt\Big].
\end{aligned}
\end{equation}
Noting that by the convexity of the lower semi-continuous functional
$$J(Z)=\dbE\int_0^T|\Delta Z^{n}_t|^2dt,$$
we have
\begin{align}\label{22.2.17.2}
 \dbE\int_0^T|\Delta Z^{n}_t|^2dt
 \les \liminf_{m\rightarrow\i}\dbE\int_0^T|\Delta Z^{m,n}_t|^2dt.
\end{align}
Hence,  combining \rf{22.2.17.3} and \rf{22.2.17.2} we have that
\begin{equation*}
\begin{aligned}
\ds\dbE\int_0^T|\Delta Z^{n}_t|^2dt
\les\dbE\Big[\Phi(\Delta Y^{n}_T)
+C_0\int_0^T\Phi'(\Delta Y^{n}_t)[1+|Z_t|^2]dt
+\frac{C}{2}\int_0^T\Phi''(\Delta Y^{n}_t)|\Delta Y^{n}_t|^2dt\Big].
\end{aligned}
\end{equation*}
By Lebesgue's dominated convergence theorem, and note that $\Phi(0)=\Phi'(0)=0$, we have that as $n \rightarrow \infty$, the right-hand side of the above inequality converges to 0, as well as the left-hand side. Now, passing to the limit as $n \rightarrow \infty$ on the above inequality, we obtain that
$$\limsup _{n \rightarrow \infty} \dbE\int_0^T|\Delta Z^{n}_t|^2dt=0,$$
which implies that the sequence  $\{Z^n;n\in\dbN\}$ strongly converges to $Z$ in  $L_{\dbF}^2(0,T;\dbR^d)$.

\ms

$\emph{Step 2.}$ The uniform convergence of a subsequence of $\{Y^n;n\in\dbN\}$ to $Y$.

\ms

At this stage we know that for all $t\in[0,T]$, ${\rm lim}_{n\rightarrow\i}Y^n_t=Y_t$, and the sequence $\{Z^n;n\in\dbN\}$ converges to $Z$ in $L_{\dbF}^2(0,T;\dbR^d)$. Then from Lemma 2.5 of Kobylanski \cite{Kobylanski-00}, there exists a subsequence $\{Z^{n_j};n_j\in\dbN\}$ of $\{Z^n;n\in\dbN\}$ such that
$Z^{n_j}$ converges almost surely to $Z$ and such that
$$\widetilde{Z}\deq\sup_j|Z^{n_j}|\in L_{\dbF}^2(0,T;\dbR^d).$$
For simplicity presentation, by otherwise choosing a subsequence, we would like to still assume that the whole sequence $\{Z^n;n\in\dbN\}$  converges almost surely to $Z$ in $L_{\dbF}^2(0,T;\dbR^d)$, and thus we have
$$Z^n\rightarrow Z\ \hb{ a.s. }dt\otimes d\dbP\q~\hb{and}\q~
\widetilde{Z}=\sup_n|Z^{n}|\in L_{\dbF}^2(0,T;\dbR^d).$$
Recall that the sequence $\{f^{n};n\in\dbN\}$ converges locally uniformly to $f$, we obtain that
$$
\lim _{n \rightarrow \infty} f^{n}\left(t, Y_{t}^{n}, Z_{t}^{n}\right)=f\left(t, Y_{t}, Z_{t}\right), \q~ t\in[0,T], \ a.s.
$$
In addition, due to that $f^n$ satisfies the condition \rf{22.2.17.4}, one has
$$
\left|f^{n}\left(t, Y_{t}^{n}, Z_{t}^{n}\right)\right| \les C(1+M+ \sup _{n}\left|Z_{t}^{n}\right|^{2})\les C_0(1+ \widetilde{Z}_t^{2}) .
$$
Thus, Lebesgue's dominated convergence theorem gives
\begin{equation*}\label{22.4.13.2}
\lim_{n\rightarrow\i}\int_0^Tf^n(t,Y^n_t,Z^n_t)dt=\int_0^Tf(t,Y_t,Z_t)dt,\q~ a.s.
\end{equation*}
Similarly, we have
\begin{equation*}
  \lim_{n\rightarrow\i}\dbE\int_0^T|g(t,Y^n_t,Z_t^n)-g(t,Y_t,Z_t)|^2dt
  \les\lim_{n\rightarrow\i}\dbE\int_0^T[C|\Delta Y^n_t|^2+\a|\Delta Z^n_t|^2]dt=0,
\end{equation*}
and therefore
$$g(t,Y^n,Z^n)\rightarrow g(t,Y,Z)\ \hb{ a.s. }dt\otimes d\dbP\q~\hb{and}\q~
\sup_n|g(t,Y^n,Z^{n})|\in L_{\dbF}^2(0,T;\dbR^d).$$
On the other hand, from the continuity properties of stochastic integral, we get that
\begin{align}
\ds  &\lim_{n\rightarrow\i}\sup_{0\les t\les T}\Big|\int_t^TZ^n_sdW_s-\int_t^TZ_sdW_s\Big|=0\q\hb{in probability,}\nn \\
\ns\ds  & \lim_{n\rightarrow\i}\sup_{0\les t\les T}\Big|\int_t^Tg(s,Y^n_s,Z^n_s)d\oaB_s-\int_t^Tg(s,Y_s,Z_s)d\oaB_s\Big|=0\q\hb{in probability.} \label{22.4.13.4}
\end{align}
Extracting a subsequence again if necessary, we may assume that the above convergence is $\dbP$-a.s.
Finally, we have
\begin{align*}
\ds |\Delta Y^{m,n}_t|\les&~ |\Delta Y^{m,n}_T|+\int_t^T|f^m(s,Y^m_s,Z^m_s)-f^n(s,Y^n_s,Z^n_s)|ds\\
\ns\ds  &+\Big|\int_t^T[g(s,Y^m_s,Z^m_s)-g(s,Y^n_s,Z^n_s)]d\oaB_s\Big|
 + \Big|\int_t^T\Delta Z^{m,n}_{s}dW_{s}\Big|.
\end{align*}
Now fix $n$, and taking limits on $m$ and supremum over $t\in[0,T]$, we obtain that for almost all $\o\in\Om$,
\begin{align*}
\ds \sup_{0\les t\les T}|\Delta Y^{n}_t|\les&\
|\Delta Y^{n}_T|+\int_t^T|f(s,Y_s,Z_s)-f^n(s,Y^n_s,Z^n_s)|ds\\
\ns\ds  &+\sup_{0\les t\les T}\Big|\int_t^T[g(s,Y_s,Z_s)-g(s,Y^n_s,Z^n_s)]d\oaB_s\Big|
 + \sup_{0\les t\les T}\Big|\int_t^T\Delta Z^{n}_{s}dW_{s}\Big|,
\end{align*}
from which we deduce that the sequence $\{Y^n;n\in\dbN\}$ converges to $Y$ uniformly for $t\in[0,T]$. In particular, $Y$ is continuous if the sequence $\{Y^n;n\in\dbN\}$ is.

\ms

Finally, we could pass to the limit in
\begin{align*}
 Y^n_t=Y^n_T+\int_t^Tf^n(s,Y^n_s,Z^n_s)ds+\int_t^Tg(s,Y^n_s,Z^n_s)d\oaB_s - \int_t^TZ^n_{s}dW_{s},\q~
 0\les t\les T,
\end{align*}
deducing that the pair $(Y,Z)$ is a solution of BDSDE with parameters $(f,g,\xi)$.
\end{proof}

\begin{remark}\rm
It should be point out that the solution here we obtain might not be unique, due to that the limit generator $f$ is of quadratic growth in $z$, but not satisfies a comparison theorem. A stability result of BDSDE will be given (see \autoref{Stability}) after the comparison theorem is proved in Section \ref{Sec4}.
\end{remark}

\subsection{Existence}

Based on the results of a priori estimate and the monotone stability, now we can prove the existence of solutions of one-dimensional BDSDEs \rf{BDSDE} with quadratic growth.

\begin{assumption}\label{H2}\rm
Suppose that the terminal value $\xi$ belongs to $L^\i_{\sF_T}(\Omega;\dbR)$, and  $c:\dbR^+\rightarrow\dbR^+$ is a continuous increasing function. Moreover, there exist some constants $a$, $b$, $C$, $d\in\dbR$ and $\a\in(0,1)$ such that for all $t\in[0,T],y,\bar y\in\dbR, z,\bar z\in\dbR^d$,
   \begin{align*}
     f(t,y,z)=a_0(t,y,z)y+f_0(t,y,z)
   \end{align*}
  with
   \begin{align*}
   \ds  &d\les a_0(t,y,z)\les a,\q~ |f_0(t,y,z)|\les b+c(|y|)|z|^2,\q~|g(t,y,z)|^2\les\a|z|^2,\q~\hb{a.s.}\\
\ns\ds  & |g(t,y_1,z_1)-g(t,y_2,z_2)|^2\les C|y_1-y_2|^2+\a|z_1-z_2|^2,\q~\hb{a.s.}
   \end{align*}
\end{assumption}

\begin{remark} \rm
Under \autoref{H2}, the coefficient $g$ is bounded with respect to $y$ and $z$. However, if the coefficient $g$ is independent of $y$, then it is of linear growth and Lipschitz continuous with respect to $z$.
In fact, \autoref{H2} implies that
\begin{equation*}\label{22.4.8.1}
  g(t,y,0)=0,\q~ \forall (t,y)\in[0,T]\ts\dbR.
\end{equation*}
Here we present two examples that satisfying \autoref{H2}:
\begin{align*}
\ds \hb{(i)}\q~  g(t,y,z)=&\ \a\sin(y)\cos(z),\q~ \forall (t,y,z)\in[0,T]\ts\dbR\ts\dbR^d;\\
\ns\ds \hb{(ii)}\q~ g(t,y,z)=&\ l(z),\q~ |l(z)|^2\les \a|z|^2,\q~ \forall (t,y,z)\in[0,T]\ts\dbR\ts\dbR^d.
\end{align*}
\end{remark}

The following theorem is the main result of this section.

\begin{theorem}[Existence]\label{existence}
Suppose that the parameters $(f,g,\xi)$ satisfy \autoref{H2}, then BDSDE \rf{BDSDE} has at least one solution $(Y,Z)\in L_{\dbF}^\i(0,T;\dbR)\times L_{\dbF}^2(0,T;\dbR^d)$, where $Y$ is continuous.
Moreover, there exists a maximal solution $(Y^*,Z^*)$ $($resp. a minimal solution  $(Y_*,Z_*))$
such that for any set of parameters  $(\hat{f},g,\hat{\xi})$, if
$$\hat{f}\les f\q\hb{and}\q\hat{\xi}\les\xi\q (\hb{resp. }\hat{f}\ges f\q\hb{and}\q\hat{\xi}\ges\xi),$$
one has
$$\hat Y\les Y^*\q (\hb{resp. }\hat Y\ges Y_*),$$
where $(\hat Y,\hat Z)$  is any solution of BDSDE with parameters $(\hat{f},g,\hat{\xi})$.
\end{theorem}

\begin{proof}
The main idea here we used is that, first,
using an exponential change and a truncation argument in order to control the growth  of the generator $f$ in $z$ and $y$, respectively. Second, finding a good approximation of $f$ and constructing a sequence of uniformly Lipschitz continuous functions. Third, combine the previous result (\autoref{PriEstimate} and \autoref{Mono-stable}) to get the solution.

\ms

First, instead of \autoref{H2}, we assume that the coefficient $f$ satisfies the following condition:
There exist some constants $\tilde a,\tilde  d\in\dbR$ and $B,C\in\dbR^+$ such that for all $(t,y,z)\in[0,T]\times\dbR\times\dbR^d$,
   \begin{align}\label{3.0}
     f(t,y,z)=a_0(t,y,z)y+f_0(t,y,z)
   \end{align}
  with
   \begin{align}\label{3.0.1}
     \tilde  d\les a_0(t,y,z)\les\tilde  a,\q~ |f_0(t,y,z)|\les B+C|z|^2,\q~\hb{a.s.}
   \end{align}
Let $(F,g,\zeta)$ be a set of parameters of BDSDE such that
\begin{equation}\label{3.7}
  F\les f\q\hb{and}\q\zeta\les\xi,
\end{equation}
and suppose that it has a solution $(\tilde{Y},\tilde{Z})\in L_{\dbF}^\i(0,T;\dbR)\times L_{\dbF}^2(0,T;\dbR^d)$. The purpose is to find a solution $(Y,Z)\in L_{\dbF}^\i(0,T;\dbR)\times L_{\dbF}^2(0,T;\dbR^d)$ of BDSDE with parameters  $(f,g,\xi)$ such that
$$\tilde Y\les Y.$$
From \autoref{PriEstimate}, if we set
$$\tilde K=(\|\xi\|_\i+BT)\exp(\tilde  a^+T),$$
then any solution $(Y,Z)$ of BDSDE with parameters  $(f,g,\xi)$ has the following estimate,
$$\|Y\|_{\i}\les\tilde K.$$
Define
\bel{3.10}M=\max(\tilde K,\|\tilde Y\|_\i).\ee

\ms

$\emph{Step 1. The exponential transformation.}$

\ms

In order to make a presentation more clearly, we rewrite BDSDE \rf{BDSDE} as follows,
\begin{align}\label{3.8}
 Y_t=\xi+\int_t^Tf(s,Y_s,Z_s)ds+\int_t^Tg(s,Y_s,Z_s)d\oaB_s - \int_t^TZ_{s}dW_{s},\q~0\les t\les T.
\end{align}
Let $\b$ be a positive constant which will be determined later.
Applying the exponential change of variable (see \autoref{Itoformula}) to $u=\exp(\beta Y)$ transforms
formally BDSDE \rf{3.8} with parameters $(f,g,\xi)$ in a BDSDE with parameters $(\bar f, \bar g,\bar\xi)$ of the following,
\begin{align}\label{3.9}
 u_t=\bar\xi+\int_t^T\bar f(s,u_s,v_s)ds+\int_t^T\bar g(s,u_s,v_s)d\oaB_s - \int_t^Tv_{s}dW_{s},\q~0\les t\les T,
\end{align}
where
\begin{equation}\label{22.2.18.1}
\begin{aligned}
  &\bar f(t,u,v)=\b uf(t,\frac{\ln u}{\b},\frac{v}{\b u})
  +\frac{1}{2}\b^2u\Big[g(t,\frac{\ln u}{\b},\frac{v}{\b u})^2-\Big(\frac{v}{\b u}\Big)^2\Big],\\
  &\bar g(t,u,v)=\b ug(t,\frac{\ln u}{\b},\frac{v}{\b u}),\qq\bar\xi=\exp(\b \xi),\qq v=\b uZ.
\end{aligned}
\end{equation}
It is easy to check that $\bar g$ and $\bar\xi$ satisfy \autoref{H2}. Similarly, the exponential change $u=\exp(\beta \tilde Y)$ transforms formally the BDSDE with parameters $(F,g,\zeta)$ in a BDSDE with parameters
$(\bar F,\bar g,\bar \zeta)$ where
\begin{align*}
  \bar F(t,u,v)=\b uF(s,\frac{\ln u}{\b},\frac{v}{\b u})
  +\frac{1}{2}\b^2u\Big[g(s,\frac{\ln u}{\b},\frac{v}{\b u})^2-\Big(\frac{v}{\b u}\Big)^2\Big],\qq \bar\zeta=\exp(\b \zeta).
\end{align*}
We consider a function $\Phi:\dbR\rightarrow[0,1]$ such that
$$\Phi(u)=\left\{\ba{ll}
\ds 1,\qq \hb{if }u\in[\exp(-\b M),\exp(\b M)],\\
\ns\ds 0,\qq \hb{if }u\notin[\exp(-\b (M+1)),\exp(\b (M+1))],
\ea\right.$$
and we set that for $(t,u,v)\in[0,T]\times\dbR\times\dbR^d$,
$$\tilde f(t,u,v)=\Phi(u)\bar f(t,u,v),\q~ \tilde F(t,u,v)=\Phi(u)\bar F(t,u,v),
\q~ l(u)=\Phi(u)\big(\tilde au\ln u+B\b u).$$
Now, let
\bel{3.9}\b=\frac{2C}{1-\a},\ee
then, note that \rf{3.0.1}, we have
$$\Phi(u)\Big(\tilde d u\ln u-B\b u-(\frac{C}{\b}+\frac{1}{2})\frac{|v|^2}{u}\Big)\les
\tilde f(t,u,v)\les l(u).$$
We remark that $l$ is a Lipschitz continuous function bounded from above by a
constant, denoted by $L$. Moreover, note that if we define
$$\tilde y=\exp(\b\tilde Y)\q~\hb{and}\q~ \tilde z=\b\tilde Z\exp(\b\tilde Y),$$
then the pair $(\tilde y,\tilde z)$ is a solution of BDSDE with parameters $(\tilde F,\bar g,\bar \zeta)$.

\ms

$\emph{Step 2. The approximation.}$

\ms

We can approximate $\tilde f$ by a decreasing sequence of uniformly Lipschitz continuous functions
$\{\tilde f^n;n\in\dbN\}$ such that for all $(t,u,v)\in[0,T]\times\dbR\times\dbR^d$,
$$\tilde F(t,u,v)\les \tilde f(t,u,v)\les \tilde f^n(t,u,v)\les l(u)+\frac{1}{2^n}.$$
Indeed, such a sequence exists. For instance,
denote by $m$ an integral, and for each integral $n\ges m$, we would like to define the sequence of the functions $\{\tilde f^n;n\in\dbN\}$ by
\begin{equation}\label{22.4.9.1}
\tilde f_{n}(t, y, z)=\sup \left\{\tilde{f}(t, p, q)-n|p-y|-n|q-z|:(p, q) \in \mathbb{Q}^{1+d}\right\},
\end{equation}
where $\dbQ$ denotes the rational number.
It is easy to see that $\tilde f_{n}$ is well defined and it is globally Lipschitz continuous in $y$ and $z$ with the constant $n$. In addition, the sequence $\{\tilde f^n;n\ges m\}$ is decreasing and converges pointwise to the function $\tilde f$.

\ms

%
%
Then the classical results of the existence and uniqueness (see \autoref{PP-EQ}) and the comparison theorem (see \autoref{Shi05}) for Lipschitz continuous coefficients give us that, for every $n\in\dbN$, there exists a unique solution $(y^n,z^n)$
of the BDSDE with parameters $(\tilde f^n,\bar g,\bar\xi)$ such that
$$\tilde y\les y^{n+1}\les y^n\les y^1.$$
Next, we indicate that the sequence $\{y^n;n\in\dbN\}$ is bounded and positive. It should be point out that by the definition of $\bar g$ (see \rf{22.2.18.1}) and \autoref{H2}, it is certainly that
\begin{equation}\label{22.4.11.6}
|\bar g(t,u,v)|^2\les \a|v|^2,\q~ \forall (t,u,v)\in[0,T]\times\dbR\times\dbR^d,
\end{equation}
which implies that $\bar g(t,u,v)=0$ when $v=0$. Hence, again, \autoref{PP-EQ} implies that, on the one hand, the pair of processes $\big\{\big(\exp(\b M), 0\big);t\in[0, T]\big\}$ is the unique solution of BDSDE with parameters $\big(0,\bar g,\exp(\b M)\big)$; on the other hand, the pair of processes $\big\{\big(\exp(-\b M), 0\big);t\in[0, T]\big\}$ is the unique solution of BDSDE with parameters $\big(0,\bar g,\exp(-\b M)\big)$.
Moreover, since for $n\in\dbN$ large enough,
$$
\exp(\b \xi)\les \exp(\b M) \quad \text { and } \quad \tilde{f}^{n}\big(t,\exp(\b M), 0\big)\les 0,\q~ t\in[0,T],
$$
and for all $n$,
$$
\exp(-\b M) \les \exp(\b \xi) \quad \text { and } \quad 0 \les \tilde{f}^{n}\big(t,\exp(-\b M), 0\big),\q~ t\in[0,T],
$$
then \autoref{Shi05} gives us that for $n\in\dbN$ large enough,
\begin{equation}\label{22.4.11.3}
\exp(-\b M) \les y^{n+1}_t\les y^n_t\les \exp(\b M),\q~ \hb{a.s.}\ t\in[0,T].
\end{equation}

Now, we come back to the previous problem by denoting
\begin{align*}
\ds  \tilde G^n(t,y,z)&=\frac{\tilde f^n(t,e^{\b y},\b e^{\b y}z)}{\b e^{\b y}}+\frac{\b}{2}\big[z^2-g(t,y,z)^2\big],\\
\ns\ds  \tilde G(t,y,z)&=\frac{\tilde f(t,e^{\b y},\b e^{\b y}z)}{\b e^{\b y}}+\frac{\b}{2}\big[z^2-g(t,y,z)^2\big]\\
&=\Phi(e^{\b y})f(t,y,z)+(1-\Phi(e^{\b y}))\frac{\b}{2}\big[z^2-g(t,y,z)^2\big].
%
\end{align*}
Notice that $\b$ is defined in \rf{3.9}. Then, the pair $(Y^n,Z^n)$ defined by
$$Y^n_t=\frac{\ln y^n_t}{\b}\q~\hb{and}\q~Z^n_t=\frac{z^n_t}{\b y^n_t}$$
is a solution of the BDSDE with parameters $(\tilde G^n,g,\xi)$. Let us recapitulate what we have obtained.
\begin{itemize}
  \item [(i)] There exist two positive constants $K$ and $C$ such that for all $(t,y,z)\in[0,T]\ts\dbR\ts\dbR^d$,
   $$|\tilde G^n(t,y,z)|\les K+C|z|^2,\q n\in\dbN.$$
  \item [(ii)] The sequence $\{\tilde G^n; n\in\dbN\}$ converges to $\tilde G$  locally uniformly on $[0,T]\ts\dbR\ts\dbR^d$.
  \item [(iii)] For each $n\in\dbN$, the BDSDE with parameters $(\tilde G^n,g,\xi)$ has a solution
   $$(Y^n,Z^n)\in L_{\dbF}^\i(0,T;\dbR)\times L_{\dbF}^2(0,T;\dbR^d),$$
  such that the sequence $\{Y^n;n\in\dbN\}$ is decreasing, and there exists a constant $\bar K>0$ such that for every $n\in\dbN$, we have that $\|Y^n\|_\i\les\bar K$.
  \item [(iv)] For each $n\in\dbN$, $\tilde Y\les Y^n$.
\end{itemize}
Therefore, applying \autoref{Mono-stable}, we have that there exists a pair of processes
$(Y,Z)\in L_{\dbF}^\i(0,T;\dbR)\times L_{\dbF}^2(0,T;\dbR^d)$ such that, the sequence $\{Y^n;n\in\dbN\}$  converges uniformly to $Y$,  the sequence $\{Z^n;n\in\dbN\}$ converges to $Z$ in $L_{\dbF}^2(0,T;\dbR^d)$, and $(Y,Z)$ is a solution of BDSDE with parameters $(\tilde G,g,\xi)$, i.e.,
\begin{align*}
 Y_t=\xi+\int_t^T\tilde G(s,Y_s,Z_s)ds+\int_t^Tg(s,Y_s,Z_s)d\oaB_s - \int_t^TZ_{s}dW_{s},\q~0\les t\les T.
\end{align*}
Finally, we declare that
$$\|Y\|_\i\les M.$$
In fact, it follows from \autoref{PriEstimate} as
\begin{align*}
     \tilde G(t,y,z)=\tilde a_0(t,y,z)y+\tilde f_0(t,y,z)
\end{align*}
with
\begin{align*}
     \tilde a_0(t,y,z)=\Phi(e^{\b y})a_0(t,y,z)\les a^+,
\end{align*}
and
\begin{align*}
     \tilde f_0(t,y,z)=\Phi(e^{\b y})f_0(t,y,z)+(1-\Phi(e^{\b y}))\frac{\b}{2}\big[z^2-g(t,y,z)^2\big],
\end{align*}
where the functions $a_0$ and $f_0$ come from \rf{3.0}. Then, notice that \rf{3.9}, we have
\begin{align*}
     |\tilde f_0(t,y,z)|\les B+\frac{3C}{1-\a}|z|^2.
\end{align*}
So, by \autoref{PriEstimate}, we obtain
$$\|Y\|_\i\les (\|\xi\|_\i+BT)\exp(a^+T)=\tilde K\les M.$$
Hence the pair $(Y,Z)$ is also a solution of the BDSDE with parameters $(f,g,\xi)$.
In addition,
$$\tilde Y\les Y\q~\hb{and}\q~\|Y\|_\i\les M.$$

\ms

$\emph{Step 3. The truncation.}$

\ms

Finally, we suppose that the generator $f$ satisfies \autoref{H2}. Let $(F,g,\zeta)$ be a set of parameters of BDSDE such that the condition \rf{3.7} hold true and denote by $(\tilde{Y},\tilde{Z})\in L_{\dbF}^\i(0,T;\dbR)\times L_{\dbF}^2(0,T;\dbR^d)$ a solution of BDSDE with the parameters $(F,g,\zeta)$. For all $(t,y,z)\in[0,T]\ts\dbR\ts\dbR^d$, we set
$$
  \check{f}(t,y,z)=a_0(t,y,z)y+f_0(t,\phi_M(y)y,z)\q~\hb{and}\q~ \check{F}=F(t,\phi_M(y)y,z),
$$
where $M$ is defined in \rf{3.10}, and
$$\phi_n(u)=\left\{\ba{ll}
\ds 1,\qq \hb{if }|u|\les n,\\
\ns\ds 0,\qq \hb{if }|u|\ges n+1.
\ea\right.$$
It is easy to check that the generator $\check f$ satisfies the condition \rf{3.0.1}, and $(\tilde{Y},\tilde{Z})$ is also a solution of the BDSDE with parameters $(\check F,g,\zeta)$. Then, by {\it Step 1} and {\it Step 2}, we deduce that  BDSDE with parameters $(\check f, g,\xi)$  has  a solution $(Y,Z)$ such that
$$\tilde Y\les Y\q~\hb{and}\q~\|Y\|_\i\les M,$$
which implies that the pair of process $(Y,Z)$ is also a solution of BDSDE with parameters $(f, g,\xi)$. then the existence of a maximal solution is proved.

\ms

Similarly, by using the same argument but with the change of variable $u=\exp(-\b y)$, one can prove the existence of a minimal solution. This completes the proof.
\end{proof}

\begin{remark} \rm
Some ideas for proving \autoref{existence} are borrowed from Kobylanski \cite{Kobylanski-00}, however, due to the emergence of the backward It\^o's integral appeared in BDSDE \rf{BDSDE}, some new ideas and technics have been introduced and used. For example, the main idea of Kobylanski \cite{Kobylanski-00} is to take the exponential change to transform a quadratic BDSDE into a classical BDSDE essentially. While, this is not enough for us, since the presence of the backward It\^o's integral. In order to overcome the troubles, we introduce the idea of the comparison theorem of classical BDSDEs, which is evident from \autoref{Exp1}.
\end{remark}


\section{Comparison Theorem}\label{Sec4}
As we all know that the one-dimensional circumstance allows us to establish a comparison theorem for the solutions of BDSDEs, which implies the uniqueness of solutions of BDSDEs as a by-product. So in this section, we would like to prove a comparison theorem for the solutions of BDSDEs with quadratic growth. For $i=1, 2$, we consider the following equation:
\begin{equation}\label{21.8.3.1}
Y_{t}^{i}=\xi^{i}+\int_{t}^{T} f^{i}\left(s,Y_s^i,Z_{s}^{i}\right) d s
+\int_{t}^{T} g\left(s,Y_s^i,Z_{s}^{i}\right) d \oaB_{s}
-\int_{t}^{T} Z_{s}^{i} d W_{s},\q~0\les t\les T.
\end{equation}
Then, under \autoref{H2}, there exists a pair of measurable processes $(Y^i,Z^i)\in L_{\dbF}^\i(0,T;\dbR)\times L_{\dbF}^2(0,T;\dbR^d)$ that satisfies the corresponding BDSDE \rf{21.8.3.1}, where $i=1,2$. Moreover, \autoref{PriEstimate} implies that there exists a positive constant $K$ such that
\begin{equation*}
\|Y^i\|_\i+\dbE\int_0^t|Z^i_s|^2ds\les K,\q ~ \forall t\in[0,T].
\end{equation*}
In order to prove the comparison theorem, we assume that the terminal value $\xi$ is bounded, and the generator $f$ is locally Lipschitz continuous and is of quadratic growth in $Z$ in a strong sense.
\begin{assumption}\label{H3}
There exist three functions $l$, $l_\e$, $k:[0,T]\rightarrow\dbR$ and three positive constants $C$, $\e$, and $\a\in(0,1)$ such that for all $t\in[0,T],\ y\in[-M,M]$ and $z\in\dbR^d$, the coefficients $f$ and $g$ satisfy the following conditions:
\begin{equation}\label{21.12.24.7}
\begin{aligned}
\ds |f(t, y, z)| & \les l(t)+C|z|^{2}, \qq\ |g(t,y,z)|^2\les \a|z|^2,
 \q~ \text { a.s., } \\
\ns\ds \frac{\partial f}{\partial y}(t,y,z)&\les l_{\varepsilon}(t)+\varepsilon|z|^{2},
\q~ \Big|\frac{\partial g}{\partial y}(t,y,z)\Big|^2\les C,
\q~\text{a.s.}, \\
\ns\ds \left|\frac{\partial f}{\partial z}(t, y, z)\right| & \les k(t)+C|z|,
\q \ \ \Big|\frac{\partial g}{\partial z}(t, y, z)\Big|^2  \les \a.
 \q~ \text { a.s. }
\end{aligned}
\end{equation}
\end{assumption}

Similar to Kobylanski \cite{Kobylanski-00}, following the method used by Barles--Murat \cite{Barles-Murat-95} for PDEs, we would like to divide the proof of the comparison theorem into two steps. First, we consider the comparison theorem under the condition that the generator $f$ satisfies a structure condition (STR).
Then, we show the proof by giving a change of variable that transforms a BDSDE with the generator $f$ satisfying \autoref{H3} into a BDSDE with the generator $\widetilde f$ satisfying this structure condition (STR).


\begin{definition}\rm
We say that a generator $f$ satisfies the structure condition (STR), if there exists a constant $a>0$ and a function $b: [0,T] \rightarrow \mathbb{R}$  such that
\begin{equation}\label{STR}
\frac{\partial f}{\partial y}(t, y, z)+a\left|\frac{\partial f}{\partial z}\right|^{2}(t, y, z) \les b(t),\q~ (t, y, z) \in [0,T] \times \mathbb{R} \times \mathbb{R}^{d}.
\end{equation}
\end{definition}

We have the following comparison theorem under the structure condition (STR).
\begin{proposition}\label{21.8.3.3}
Assume that both $\xi^1$ and $\xi^2$ are bounded, the coefficient $g$ satisfies \autoref{H3}, and either $f^1$ or $f^2$ satisfies the structure condition (STR) with the constant $a>0$ and the function $b\in L^1([0,T];\dbR)$.
For $i=1,2$, denote by $(Y^i,Z^i)$ a  solution of BDSDE \rf{21.8.3.1} with parameters $(f^i,g,\xi^i)$, respectively.
Moreover, suppose that for all $(t,y,z)\in[0,T]\times\dbR\times\dbR^d$,
\begin{equation}\label{21.12.24.6}
 \xi^1\les \xi^2,\q~ f^1(t,y,z)\les f^2(t,y,z),\q~ \text{a.s.},
\end{equation}
then for any $t\in[0,T]$, we have
\begin{align*}
   Y_t^1\les Y_t^2,\q \hb{a.s.}
\end{align*}
%
%
\end{proposition}

\begin{remark} \rm
(i) It should be point out that in the condition (STR), when $a>1/(p-1)(1-\a)$ and the function $b$ is bounded, \autoref{21.8.3.3} holds true in the space $L_{\dbF}^p(0,T;\dbR)\times L_{\dbF}^2(0,T;\dbR^d)$.

\ms

(ii) \autoref{21.8.3.3} still holds true if $f^{1}\left(t, Y_{t}^{2}, Z_{t}^{2}\right) \les f^{2}\left(t, Y_{t}^{2}, Z_{t}^{2}\right)$ a.s. for all $t\in[0,T]$ and the generator $f^{1}$ satisfies the condition (STR), or if either $f^{1}\left(t, Y_{t}^{1}, Z_{t}^{1}\right) \les f^{2}\left(t, Y_{t}^{1}, Z_{t}^{1}\right)$ a.s. for all $t\in[0,T]$ and the generator $f^{2}$ satisfies the condition (STR).  \rm
\end{remark} \rm

\begin{proof}[Proof of \autoref{21.8.3.3}]
For simplicity of presentation, we denote that for $t\in[0,T]$,
$$
\widehat{Y}_t= Y^1_t-Y^2_t\q~\hb{and}\q~ \widehat{Z}_t=Z^1_t-Z^2_t,$$
and for any $s\in[t,T]$,
$$\widehat f(s)= f^1(s,Y^1_s,Z^1_s)-f^2(s,Y^2_s,Z^2_s)\q~\hb{and}\q~
\widehat g(s)=g(s,Y^1_s,Z^1_s)-g(s,Y^2_s,Z^2_s).$$
Then, note that $\widehat{Y}^+=\max\{0,\widehat{Y}\}$,
%
using It\^{o}'s formula (see \autoref{Itoformula}) to $\big(\widehat{Y}^+\big)^p$ for $p\in\dbN$ with $p>2$ implies that
\begin{equation}\label{21.8.3.0}
\begin{aligned}
\ds  (\widehat{Y}_{t}^{+})^{p}
=&\ p\int_{t}^{T} (\widehat{Y}^{+}_s)^{p-1}\widehat  f(s)ds
+p\int_{t}^{T} (\widehat{Y}^{+}_s)^{p-1}\widehat  g(s)d\oaB_s
-p\int_{t}^{T} (\widehat{Y}^{+}_s)^{p-1}\widehat Z_sdW_s\\
\ns\ds &
+\frac{p(p-1)}{2}\int_t^T(\widehat{Y}^{+}_s)^{p-2}|\widehat  g(s)|^2ds
-\frac{p(p-1)}{2}\int_t^T(\widehat{Y}^{+}_s)^{p-2}|\widehat Z_s|^2ds.
\end{aligned}
\end{equation}
For the term $\widehat f$ in the above, note \rf{21.12.24.6}, we have that
$$\begin{aligned}
\ds \widehat f(s)=&~f^1(s,Y^1_s,Z^1_s)-f^2(s,Y^2_s,Z^2_s)\\
\ns\ds =&~f^1(s,Y^1_s,Z^1_s)-f^2(s,Y^1_s,Z^1_s)+f^2(s,Y^1_s,Z^1_s)- f^2(s,Y^2_s,Z^2_s)\\
\ns\ds \les&~f^2(s,Y^1_s,Z^1_s)- f^2(s,Y^2_s,Z^2_s)\\
\ns\ds \les&~ \Big(\int_0^1\frac{\partial f^2}{\partial y}(*)d\l\Big) \widehat{Y}_s
+\Big(\int_0^1\frac{\partial f^2}{\partial z}(*)d\l\Big) \widehat{Z}_s,
\end{aligned}$$
where
$$(*)=(s,\l Y^1_s+(1-\l)Y^2_s,\l Z^1_s+(1-\l)Z^2_s).$$
Then, for the first term in the right hand side of \rf{21.8.3.0}, we have that
\begin{equation}\label{21.12.24.2}\begin{aligned}
\ds (\widehat{Y}^{+}_s)^{p-1}\widehat  f(s)
\les (\widehat{Y}^{+}_s)^{p}\int_0^1\bigg(\frac{\partial f^2}{\partial y}
+a\left|\frac{\partial f^2}{\partial z}\right|^2\bigg)(*)d\l
+\frac{1}{4a}(\widehat{Y}^{+}_s)^{p-2}| \widehat{Z}_s|^2{\bf 1}_{\{\widehat{Y}^{+}\ges0\}},
\end{aligned}\end{equation}
where we have used the well-known inequality $2\beta\g\les |\b|^2+|\g|^2$ with
$$
\begin{aligned}
\ds &\b=\sqrt{2 a} \frac{\partial f^{2}}{\partial z}(*)\left(Y_{s}^{+}\right)^{p / 2}, \\
\ns\ds &\g=\frac{1}{\sqrt{2 a}}\left(Y_{s}^{+}\right)^{(p-2) / 2} Z_{s} \mathbf{1}_{\left\{Y^{+} \ges 0\right\}}.
\end{aligned}
$$
For the term $\widehat{g}$ in \rf{21.8.3.0}, note \autoref{H3} satisfied by the coefficient $g$, we have that
$$\begin{aligned}
\ds |\widehat g(s)|^2=&~ \bigg\{\Big(\int_0^1\frac{\partial g}{\partial y}(*)d\l\Big) \widehat{Y}_s
+\Big(\int_0^1\frac{\partial g}{\partial z}(*)d\l\Big) \widehat{Z}_s\bigg\}^2\\
\ns\ds \les &~ \Big(\int_0^1\frac{\partial g}{\partial y}(*)d\l\Big)^2 |\widehat{Y}_s|^2
+2\frac{\sqrt{\e}}{\sqrt{\e}}\Big(\int_0^1\frac{\partial g}{\partial y}(*)d\l\Big)\Big(\int_0^1\frac{\partial g}{\partial z}(*)d\l\Big) |\widehat{Y}_s\widehat{Z}_s|
+\Big(\int_0^1\frac{\partial g}{\partial z}(*)d\l\Big)^2 |\widehat{Z}_s|^2\\
\ns\ds \les&~ \widetilde{C}|\widehat{Y}_s|^2+(\a+\e) |\widehat{Z}_s|^2,
\end{aligned}$$
where
$$\widetilde{C}=C+\frac{C\a}{\e}\q~\hb{and}\q~ \e=\frac{1-\a}{2}.$$
So, for the fourth term on the right hand side of \rf{21.8.3.0}, we have that
\begin{equation}\label{21.12.24.3}
\begin{aligned}
(\widehat{Y}^{+}_s)^{p-2}|\widehat  g(s)|^2
\les \widetilde{C}(\widehat{Y}^{+}_s)^{p}
+\frac{1+\a}{2}(\widehat{Y}^{+}_s)^{p-2}| \widehat{Z}_s|^2{\bf 1}_{\{\widehat{Y}^{+}\ges0\}}.
\end{aligned}
\end{equation}
Combining \rf{21.8.3.0}-\rf{21.12.24.3} and (STR), one has
\begin{equation}\label{21.12.24.4}
\begin{aligned}
\ds  &(\widehat{Y}_{t}^{+})^{p}+\frac{p}{4}\left((p-1)(1-\a)-\frac{1}{a}\right)
\int_t^T(\widehat{Y}^{+}_s)^{p-2}|\widehat Z_s|^2{\bf 1}_{\{\widehat{Y}^{+}\ges0\}}ds\\
\ns\ds \les&\ p\int_{t}^{T} \big[b(s)+\frac{1}{2}(p-1)\widetilde{C}\big](\widehat{Y}^{+}_s)^{p}ds
+p\int_{t}^{T} (\widehat{Y}^{+}_s)^{p-1}\widehat  g(s)d\oaB_s
-p\int_{t}^{T} (\widehat{Y}^{+}_s)^{p-1}\widehat Z_sdW_s,
\end{aligned}
\end{equation}
where the function $b$ comes from \rf{STR}.
Note that $\widehat Y$ is bounded, which implies that $(\widehat{Y}^{+}_s)^{p-1}\widehat  g(s)$ and $(\widehat{Y}^{+}_s)^{p-1}\widehat Z_s$ belong to the space $L^2_{\dbF}(0,T;\dbR^d)$. Take the expectation on both side of \rf{21.12.24.4},
\begin{equation*}
\begin{aligned}
\ds &\dbE(\widehat{Y}_{t}^{+})^{p}+\frac{p}{4}\left((p-1)(1-\a)-\frac{1}{a}\right)
\dbE\int_t^T(\widehat{Y}^{+}_s)^{p-2}|\widehat Z_s|^2{\bf 1}_{\{\widehat{Y}^{+}\ges0\}}ds\\
\ns\ds\les&\ p\int_{t}^{T} \big[b(s)+\frac{1}{2}(p-1)\widetilde{C}\big](\widehat{Y}^{+}_s)^{p}ds.
\end{aligned}
\end{equation*}
It is easy to choose a large enough $p$ such that
$$\frac{p}{4}\left((p-1)(1-\a)-\frac{1}{a}\right)\ges0.$$
Finally, Gronwall's inequality deduce that
\begin{equation*}
\dbE(\widehat{Y}_{t}^{+})^{p}\les0,\q~ \forall t\in[0,T].
\end{equation*}
Therefore, for all $t\in[0,T]$ we have that
$$Y^1_t\les Y^2_t,\q~ a.s.$$
This completes the proof.
\end{proof}


Based on the first step concerning the structure condition (STR), next we are going to prove the comparison theorem under \autoref{H3}.



\begin{theorem}[Comparison theorem]\label{21.12.24.8}
Let $(\xi^1,f^1,g)$ and $(\xi^2,f^2,g)$ be two parameters of BDSDE \rf{21.8.3.1} with bounded terminal values, and suppose that
\begin{itemize}
  \item [{\rm (i)}]  $\xi^1\les \xi^2$, a.s., and $ f^1(t,y,z)\les f^2(t,y,z)$, a.s., for all
  $(t,y,z)\in[0,T]\times[-M,M]\times\dbR^d.$
  \item [{\rm (ii)}] Either $(f^1,g)$ or $(f^2,g)$ satisfies \autoref{H3}.
\end{itemize}
Then if $(Y^1,Z^1)$ and $(Y^2,Z^2)$ are the associated solutions of BDSDE \rf{21.8.3.1} with parameters $(\xi^1,f^1,g)$ and $(\xi^2,f^2,g)$, respectively, one has that for any $t\in[0,T]$,
\begin{align*}
   Y_t^1\les Y_t^2,\q~ a.s.
\end{align*}
\end{theorem}

\begin{remark} \rm
We point out that \autoref{21.12.24.8} still holds true if $f^{1}\left(t, Y_{t}^{2}, Z_{t}^{2}\right) \les f^{2}\left(t, Y_{t}^{2}, Z_{t}^{2}\right)$ a.s. for all $t\in[0,T]$ and the generator $f^{1}$ satisfies \autoref{H3}, or if either $f^{1}\left(t, Y_{t}^{1}, Z_{t}^{1}\right) \les f^{2}\left(t, Y_{t}^{1}, Z_{t}^{1}\right)$ a.s. for all $t\in[0,T]$ and the generator $f^{2}$ satisfies \autoref{H3}.
\end{remark} \rm

\begin{proof}[Proof of \autoref{21.12.24.8}]
The main idea of the proof is to look for a change of variable,  which could transform a coefficient $f$ that satisfies \autoref{H3} into a coefficient $\widetilde f$ that satisfies the structure condition (STR), and then use \autoref{21.8.3.3} to getting the goal.

\ms

Let $(Y, Z)\in L_{\dbF}^\i(0,T;\dbR)\times L_{\dbF}^2(0,T;\dbR^d)$ be a solution of BDSDE  \rf{21.8.3.1} with parameters $(\xi,f,g)$, where $\xi$ is a bounded terminal value.
Take $M \in \mathbb{R}$ such that $\|Y\|_{\infty}<M$, and consider the change of variable $y=\phi(\widetilde{y})$, where $\phi$ is a regular increasing function yet to be chosen.
Denote
\begin{align*}
\ds Y=\phi(\widetilde{Y}),\q~w(Y)=\phi'(\widetilde{Y}),\q~ Z=\phi'(\widetilde{Y})\widetilde{Z}=w(Y)\widetilde{Z},
\end{align*}
then
\begin{align*}
\widetilde{Y}=\phi^{-1}(Y),\q~ \widetilde{Z}=\frac{Z}{\phi'(\widetilde{Y})}=\frac{Z}{w(Y)},
\end{align*}
where $(\widetilde Y, \widetilde Z)$ is a  solution of the following BDSDE with parameters $(\widetilde{\xi},\widetilde f,\widetilde g)$:
\begin{equation}\label{21.12.29.1}
\widetilde Y_{t}=
\widetilde\xi+\int_{t}^{T} \widetilde f(s,\widetilde Y_s,\widetilde Z_{s}) d s
+\int_{t}^{T}\widetilde g(s,\widetilde Y_s,\widetilde Z_{s}) d \oaB_{s}
-\int_{t}^{T}\widetilde  Z_{s} d W_{s},
\end{equation}
where
\begin{equation}\label{21.12.25.1}
  \begin{aligned}
\ds \widetilde\xi&=\phi^{-1}(\xi),\\
\ns\ds \widetilde g(t,\widetilde y,\widetilde z)&=\frac{g(t,y,z)}{\phi'(\widetilde{y})}
=\frac{g(t,y,z)}{w(y)}
=\frac{g\big(t,\phi(\widetilde y), \phi'(\widetilde{y})\widetilde z \big)}{\phi'(\widetilde{y})}
,\\
\ns\ds \widetilde f(t, \widetilde y, \widetilde z)&=\frac{1}{\phi^{\prime}(\widetilde y)}\left(f\left(t, \phi(\widetilde y), \phi^{\prime}(\widetilde y) \widetilde z\right)+\frac{1}{2} \phi^{\prime \prime}(\widetilde y)\left[ \widetilde z^{2}-\widetilde g(t,\widetilde y,\widetilde z)^2\right]\right).
%
  \end{aligned}
\end{equation}
In fact, applying It\^o's formula to $\phi(\widetilde{Y})$, we get that
\begin{equation*}
\begin{aligned}
\ds \phi(\widetilde Y_t)=&\ \phi(\widetilde \xi)
+\int_{t}^{T}\left[ \phi'(\widetilde Y_s)\widetilde f(s,\widetilde Y_s,\widetilde Z_{s})
+\frac{1}{2}\phi''(\widetilde Y_s) \widetilde g(s,\widetilde Y_s,\widetilde Z_{s})^2
-\frac{1}{2}\phi''(\widetilde Y_s)\widetilde Z_s^2\right] d s\\
\ns\ds &\ +\int_{t}^{T}\phi'(\widetilde Y_s) \widetilde g(s,\widetilde Y_s,\widetilde Z_{s}) d \oaB_{s}
   -\int_{t}^{T}\phi'(\widetilde Y_s)\widetilde  Z_{s} d W_{s},
\end{aligned}
\end{equation*}
which, note that $Y=\phi(\widetilde{Y})$, comparing with BDSDE \rf{21.8.3.1} implies \rf{21.12.25.1}. In other words, BDSDEs \rf{21.8.3.1} and \rf{21.12.29.1} are equivalent. It is easy to check that the coefficient $\widetilde g$ satisfies \autoref{H3}.
Next, we are going to verify that the coefficient $\widetilde f$ given by \rf{21.12.25.1} satisfies the structure condition (STR).
Note that
\begin{align*}
\ds &y=\phi(\widetilde{y}),\q~
w(y)=\phi'(\widetilde y)=w(\phi(\widetilde y)),\q~
z=\phi'(\widetilde y)\widetilde z=w(y)\widetilde z=w(\phi(\widetilde y))\widetilde z,\\
\ns\ds &\frac{\partial y}{\partial\widetilde y}=w,\q~
\frac{\partial z}{\partial\widetilde y}=w'w\widetilde z,\q~
\frac{\partial z}{\partial\widetilde z}=w,\\
\ns\ds & \phi''(\widetilde y)=\frac{\partial w}{\partial \widetilde y}=\frac{\partial w}{\partial y}\cd \frac{\partial y}{\partial\widetilde y}=w'w,\q~
\frac{\partial w'}{\partial \widetilde y}=\frac{\partial w'}{\partial y}\cd \frac{\partial y}{\partial\widetilde y}=w''w,\\
\ns\ds& \frac{\partial \widetilde g}{\partial \widetilde{y}}(t, \widetilde {y}, \widetilde {z})
=\frac{\partial \widetilde g}{\partial y}(t, \widetilde {y}, \widetilde {z})
  \cd \frac{\partial y}{\partial\widetilde y}
 =\frac{\partial \big(g(t,y,z)/w\big)}{\partial y}\cd \frac{\partial y}{\partial\widetilde y}
 =\frac{\partial g}{\partial y}(t,y,z)-\frac{w'}{w}g(t,y,z),\\
\ns\ds& \frac{\partial \widetilde g}{\partial \widetilde{z}}(t, \widetilde {y}, \widetilde {z})
=\frac{\partial \widetilde g}{\partial z}(t, \widetilde {y}, \widetilde {z})
  \cd \frac{\partial z}{\partial\widetilde z}
 =\frac{\partial \big(g(t,y,z)/w\big)}{\partial z}\cd \frac{\partial z}{\partial\widetilde z}
 =\frac{\partial g}{\partial z}(t,y,z).
%
%
\end{align*}
From \rf{21.12.25.1}, we could compute that
%
\begin{align*}
\ds \frac{\partial \widetilde f}{\partial \widetilde{y}}(t, \widetilde {y}, \widetilde {z})
&=-\frac{w'}{w}\left(f\left(t, y, z\right)+\frac{1}{2} w'w
[\widetilde z^{2}-\widetilde{g}(t,\widetilde y,\widetilde{z})^2]\right)\\
\ns\ds &\q\ +\frac{1}{w}\bigg(
 \frac{\partial f}{\partial y}(t, y, z)w
+\frac{\partial f}{\partial z}(t, y, z)w'w\widetilde{z}
+\frac{1}{2}[ w''w^2+(w')^2w]\cd [\widetilde z^{2}-\widetilde{g}(t,\widetilde y,\widetilde{z})^2]\\
\ns\ds&\q\ +w'w\widetilde g(t,\widetilde y,\widetilde{z})\[\frac{w'}{w}g(t,y,z)
-\frac{\partial g}{\partial y}(t,y,z)\]\bigg)\\
%
%
\ns\ds &= -\frac{w'}{w}\left(f\left(t, y, z\right)+\frac{1}{2}\frac{w'}{w}[z^2-g(t,y,z)^2]\right)\\
\ns\ds &\q\ +\frac{\partial f}{\partial y}(t, y, z)
+\frac{w'}{w}\frac{\partial f}{\partial z}(t, y, z)z
+\frac{1}{2} \Big[\frac{w''}{w}+\Big(\frac{w'}{w}\Big)^2\Big]\cd [z^2-g(t,y,z)^2] \\
\ns\ds&\q\ +\(\frac{w'}{w}\)^2g(t,y,z)^2-\frac{w'}{w}g(t,y,z)\cd \frac{\partial g}{\partial y}(t,y,z)\\
\ns\ds &= \frac{1}{2}\frac{w''}{w} z^2
+\frac{w'}{w} \left( \frac{\partial f}{\partial z}(t, y, z)z - f(t, y, z)\right)
+\frac{\partial f}{\partial y}(t, y, z)  \\
\ns\ds&\q\ +\[\(\frac{w'}{w}\)^2-\frac{1}{2}\frac{w''}{w}\] g(t,y,z)^2-\frac{w'}{w}g(t,y,z)\cd \frac{\partial g}{\partial y}(t,y,z),\\
\ns\ds \frac{\partial \widetilde f}{\partial \widetilde{z}}(t, \widetilde {y}, \widetilde {z})
&=\frac{\partial f}{\partial z}(t, y, z)
+ \frac{w^{\prime}}{w}\[z-g(t,y,z)\cd \frac{\partial g}{\partial z}(t,y,z)\].
\end{align*}
%
We now show that a good choice of $\phi$ allows $\widetilde f$ to satisfy the structure condition (STR). Indeed, if $\phi$ is such that $w>0$, $w^{\prime}>0$, and $w''<0$, note that \rf{21.12.24.7} with $0<\a^2<\a<1$, then
%
\begin{align*}
\ds &\left(\frac{\partial\widetilde{f}}{\partial \widetilde{y}}+a\left|\frac{\partial \widetilde{f}}{\partial \widetilde{z}}\right|^{2}\right)(t, \widetilde{y}, \widetilde{z}) \\
\ns\ds &=\frac{1}{2}\frac{w''}{w} z^2
+\frac{w'}{w} \Big( \frac{\partial f}{\partial z}z - f\Big)
+\[\(\frac{w'}{w}\)^2-\frac{1}{2}\frac{w''}{w}\] g^2-\frac{w'}{w}g\cd \frac{\partial g}{\partial y}\\
\ns\ds&\q +\frac{\partial f}{\partial y}
+a\left|\frac{\partial f}{\partial z}+\frac{w^{\prime}}{w}
\[z-g\cd \frac{\partial g}{\partial z}\]\right|^{2} \\
\ns\ds & \les \[\a\(\frac{w'}{w}\)^2+\frac{1-\a}{2}\frac{w''}{w}\]|z|^{2}
+\frac{w'}{w}\Big(k(t)|z|+l(t)+2 C|z|^{2}\Big)
+\[\(\frac{w'}{w}\)^2-\frac{1}{2}\frac{w''}{w}\]l(t)\q\\
\ns\ds &\q +\frac{w'}{w}C\a|z|
+l_{\varepsilon}(t)+\varepsilon|z|^{2}+a\left(k(t)+\(C+\frac{w^{\prime}}{w}(1+\a)\)|z|\right)^{2}\\
\ns\ds &\les|z|^{2}\left[\a\(\frac{w'}{w}\)^2+\frac{1-\a}{2}\frac{w''}{w}
+\frac{w^{\prime}}{w}2C +\varepsilon
+a\(C+\frac{w^{\prime}}{w}(1+\a)\)^{2}\right] \\
\ns\ds &\q +|z|\left[\frac{w^{\prime}}{w} \[k(t)+C \a \]
+2 a k(t)\(C+\frac{w^{\prime}}{w}(1+\a)\)\right]\\
\ns\ds&\q +\[\(\frac{w'}{w}\)^2-\frac{1}{2}\frac{w''}{w}\]l(t)
+\frac{w^{\prime}}{w} l(t)+l_{\varepsilon}(t)+ak(t)^{2} \\
\ns\ds &\les |z|^{2}\left[
\frac{1-\a}{2}\frac{w''}{w}
+\frac{w^{\prime}}{w}2C
+\Big(\frac{w'}{w}\Big)^2
+\varepsilon
+2a\left(C+\frac{w^{\prime}}{w}(1+\a)\right)^{2}\right] \\
\ns\ds&\q +\[\(\frac{w'}{w}\)^2-\frac{1}{2}\frac{w''}{w}\]l(t)
+\frac{w^{\prime}}{w} l(t)+l_{\varepsilon}(t)+ak(t)^{2}
 +\frac{1}{4(1-\a)}\[k(t)+C\a \]^2+ak(t)^2,
\end{align*}
%
where in the last inequality we have used the inequality $2\beta_i\g_i\les |\b_i|^2+|\g_i|^2$ for $i=1,2$, with
$$
\begin{aligned}
 \b_1&=\sqrt{1-\a}\frac{w'}{w}|z|, \q~ \g_1=\frac{1}{2\sqrt{1-\a}}\[k(t)+C\a \];\\
\b_2&= |z| \sqrt{a} \left(C+\frac{w^{\prime}}{w}(1+\a)\right),\q~ \g_2=\sqrt{a} k(t).
\end{aligned}
$$
Thus, if we find $\phi$ satisfying all the required assumptions and such that on $[-M, M]$,
\begin{equation*}\label{21.12.25.6}
\frac{1-\a}{2}\frac{w''}{w}
+\frac{w^{\prime}}{w}2C
+\Big(\frac{w'}{w}\Big)^2
<-\d<0,
\end{equation*}
then choosing $a$ and $\varepsilon$ small enough, the coefficient before $|z|^{2}$ is non-positive for all $y\in[-M,M]$. Therefore (STR) is satisfied. In fact, by setting
$$
\phi(y)=\frac{1}{\lambda} \ln \left(\frac{e^{\lambda A \widetilde{y}}+1}{A}\right)-M,
$$
then a straightforward yet tedious computation gives us that
$$w=A-\exp\{-\lambda(y+M)\},\q~ w'=\l\exp\{-\lambda(y+M)\},\q~ w''=-\l^2\exp\{-\lambda(y+M)\},$$
which implies that $w''<0$ for all $y\in[-M,M]$ and $\l>0$. Moreover, when $A>1$ and $\lambda>0$, it is easy to see that $w>0$ and $w^{\prime}>0$ on the interval $[-M, M]$. Furthermore, one could choose a proper $A$ and $\lambda$ such that
$$
\begin{aligned}
\ds &\frac{1-\a}{2} \frac{w^{\prime \prime}}{w} +\frac{w^{\prime}}{w} 2 C
+\left(\frac{w^{\prime}}{w}\right)^{2} \\
\ns\ds &=\frac{1}{2}\frac{\exp (-\lambda(u+M))}{(A-\exp (-\lambda(u+M)))^{2}} \\
\ns\ds &\q \times\Big[\lambda^{2}\Big(-(1-\a)A+(3-\a) \exp (-\lambda(u+M))\Big)
+\lambda 4 C(A-\exp (-\lambda(u+M)))\Big]\\
\ns\ds&\les-\d<0.
\end{aligned}
$$
The proof is complete.
\end{proof}

As an important application of the comparison theorem, we study the stability of BDSDEs.

\begin{theorem}[Stability]\label{Stability}
Let $(f,g,\xi)$ be a set of parameters that satisfying the assumptions of \autoref{21.12.24.8}, and let $(f^n,g,\xi^n)$ be a sequence of parameters of BDSDEs such that:
\begin{itemize}
  \item [\rm{(i)}] There exist some constants $a$, $b$, $C$, $d\in\dbR$, $\a\in(0,1)$, and a continuous increasing function  $c:\dbR^+\rightarrow\dbR^+$ such that for all $n\in\dbN$, the coefficients $(f^n,g,\xi^n)$ satisfies \autoref{H2}.
  \item  [\rm{(ii)}] For all $n\in\dbN$, there exists a solution $\left(Y^{n}, Z^{n}\right)$ to the BDSDE with parameters $(f^n,g,\xi^n)$.
  \item  [\rm{(iii)}] The sequence $\{f^n;n\in\dbN\}$ converges to $f$ locally uniformly on $[0,T]\times \mathbb{R} \times \mathbb{R}^{d}$, and the sequence $\{\xi^n;n\in\dbN\}$ converges to $\xi$ in $L^{\infty}(\Omega)$.
\end{itemize}
Then there exists a pair of processes $(Y,Z)\in L_{\dbF}^\i(0,T;\dbR)\times L_{\dbF}^2(0,T;\dbR^d)$ such that
\begin{equation*}
\left\{\begin{aligned}
  &\lim_{n\rightarrow\i}Y^n=Y \hb{ uniformy on } [0,T],\\
  &\{Z^n;n\in\dbN\} \hb{ converges to $Z$ in } L_{\dbF}^2(0,T;\dbR^d).
\end{aligned}\right.
\end{equation*}
Moreover, the pair $(Y,Z)$ is a solution of BDSDE with parameters $(f,g,\xi)$.
\end{theorem}

\begin{proof}
For simplicity presentation, we define
\begin{align*}
\ds H_{n}=&\ \inf _{m \ges n} f^{m},\q~ H^{n}=\sup _{m \ges n} f^{m}, \\
\ns\ds  \xi_{n*}=&\ \inf _{m \ges n} \xi^{m},\q~ \xi^{n *}=\sup _{m \ges n} \xi^{m},
\end{align*}
and we consider the minimal solutions $\left(Y_{*}^{n} Z_{*}^{n}\right)$ of BDSDE with parameters $\left(H^{n}, \xi^{n*}\right)$, and the maximal solutions $\left(Y^{n *}, Z^{n *}\right)$ of BDSDE with parameters $\left(H_{n}, \xi_{n*}\right)$, as both:
\begin{itemize}
  \item [\rm{(i)}] The sequence $\{H_{n};n\in\dbN\}$ is increasing and converges locally uniformly to $f$, and the sequence $\{\xi_{n*};n\in\dbN\}$ is increasing.
  \item [\rm{(ii)}] The sequence $\{H^n;n\in\dbN\}$ is decreasing and converges locally uniformly to $f$, and the sequence $\{\xi^{n \ast};n\in\dbN\}$ is decreasing.
\end{itemize}
Then we have that
\begin{itemize}
  \item [\rm{(i)}] The sequence $\{Y^{n *};n\in\dbN\}$ is bounded and decreasing, and for every $n \in \mathbb{N}$
      $$Y^{n *} \ges Y^{n}.$$
      Hence, from \autoref{Mono-stable}, there exists $\left(Y^{*}, Z^{*}\right)$ such that $\{Y^{n *};n\in\dbN\}$ converges uniformly to $Y^{*}$, and $(Y^{*}, Z^{*})$ is a solution of the BDSDE with parameters $(f, g, \xi)$.
  \item [\rm{(ii)}] The sequence $\{Y_{*}^{n};n\in\dbN\}$ is bounded and decreasing, and for every $n \in \mathbb{N}$,
      $$Y_{*}^{n} \les Y^{n}.$$
      Hence, from  \autoref{Mono-stable}, there exists $\left(Y_{*}, Z_{*}\right)$ such that $\{Y_{*}^{n};n\in\dbN\}$ converges uniformly to $Y_{*}$, and $\left(Y_{*}, Z_{*}\right)$ is a solution of the BDSDE with parameters $(f, g, \xi)$.
\end{itemize}
 Finally, by \autoref{21.12.24.8}, for every $n \in \mathbb{N}$, one has that
$$Y_{*}^{n} \les Y^{n} \les Y^{n *} \quad~ \text {and} \quad~ Y_{*}=Y^{*}=Y,$$
therefore the sequence $\{Y^{n};n\in\dbN\}$ converges uniformly to $Y$. This completes the proof.
\end{proof}

\section{Application to SPDE}\label{Sec5}

When the generator $f$ is of linear growth with respect to $y$ and $z$, Pardoux--Peng \cite{Pardoux-Peng-94} used BDSDEs to give a probabilistic representation for the classical solution of semilinear SPDEs; and Bally--Matoussi \cite{Bally-Matoussi-01} and Zhang--Zhao \cite{Zhang-Zhao-07} obtained the relationship between the solution of BDSDEs and the Sobolev solution of SPDEs.
Then, Wu--Zhang \cite{Wu-Zhang-11} got the Sobolev solution of related SPDEs when $f$ is continuous and  locally monotone in $y$.
Zhang--Zhao \cite{Zhang-Zhao-15} and Bahlali et al. \cite{Bahlali-17} used BDSDEs to prove the existence and uniqueness of related SPDEs when the generator $f$ is of polynomial growth in $Y$ and
grows in $Z$ super-linearly (or sub-quadratically), respectively.
In this section, when the generator $f$ is of quadratic growth in $z$, we use BDSDEs to give a probabilistic representation for the solutions of related semilinear SPDEs in Sobolev spaces, and use it to prove the existence and uniqueness of Sobolev solutions of the SPDEs, thus extending the nonlinear stochastic Feynman-Kac formula.

\ms

First, we recall some notations. For Euclidean spaces $\dbH$ and $\dbG$, denote by $C_{l, b}^{k}(\dbH;\dbG)$ the set of functions of class $C^{k}$ from $\dbH$ to $\dbG$, whose partial derivatives of order less than or equal to $k$ are bounded.
Denote by $\mathcal{C}_{c}^{1, \infty}\left([0, T] \times \dbH\right)$ the set of compactly supported functions $\varphi(t, x)$ which are continuously derivable in the $t$-variable and infinitely continuously derivable in the $x$-variable.

\ms

Consider the following forward-backward doubly stochastic differential equation:
\begin{align}
\ds X^{t,x}_s=&\ x+\int_t^s b(X^{t,x}_r)dr+\int_t^s\sigma(X^{t,x}_r)dW_r,\q~ t\les s\les T, \label{FSDE}\\
\ns\ds Y^{t,x}_s=&\ h(X^{t,x}_T)+\int_s^Tf(r,X^{t,x}_r,Y^{t,x}_r,Z^{t,x}_r)dr
+\int_s^Tg(r,X^{t,x}_r,Y^{t,x}_r,Z^{t,x}_r)d\oaB_r - \int_s^TZ^{t,x}_{r}dW_{r}, \label{FBDSDE}
\end{align}
where the coefficients $b$ and $\sigma$ come from $C_{l, b}^{2}(\dbR^n;\dbR^n)$ and $C_{l, b}^{3}(\dbR^n;\dbR^{n\times d})$, respectively. Then it is well known that the forward equation \rf{FSDE} admits a unique adapted solution, denoted by $\{X^{t,x}_s;t\les s\les T\}$, which satisfying
$$\dbE\left[\sup_{t\les s\les T}|X^{t,x}_s|^p \right]<\i,\q~ \forall p>1.$$
Now, we would like to connect the forward-backward system \rf{FSDE} and \rf{FBDSDE} with quadratic growth to the following semilinear stochastic partial differential equation:
\begin{equation}\label{SPDE}
\begin{aligned}
\ds u(s, x)=& h(x)+\int_{s}^{T}\left\{\mathcal{L} u(r, x)+f\big(r, x, u(r, x),(\sigma^\top \nabla u)(r, x)\big)\right\} dr \\
\ns\ds &+\int_{s}^{T} g\big(r, x, u(r, x),(\sigma^\top \nabla u)(r, x)\big)d\oaB_r, \quad t \leqslant s \leqslant T,
\end{aligned}
\end{equation}
where $\sigma^\top$ denotes the transposed of $\sigma$, and
\begin{align}\label{Lu}
\mathcal{L}=\sum_{i=1}^{n} b_{i} \frac{\partial}{\partial x_{i}}+\frac{1}{2} \sum_{i, j=1}^{n} a_{i j} \frac{\partial^{2}}{\partial x_{i} \partial x_{j}}, \quad~\left(a_{i j}\right)=\sigma \sigma^\top.
\end{align}
Before introducing the definition of Sobolev solutions of SPDE \rf{SPDE}, we let
$\rho: \mathbb{R}^{n} \rightarrow \mathbb{R}^{+}$ be an integrable continuous positive function, and $L^{2}(\mathbb{R}^{n}; \rho^{-1}(x) \mathrm{d} x)$ be the weighted $L^{2}$ space endowed with the norm
$$
\|u\|_{\rho}^{2}\deq\int_{\mathbb{R}^{n}}|u(x)|^{2} \rho^{-1}(x)d x.
$$
Let us take the weight $\rho(x)=\exp\{F(x)\}$, where $F: \mathbb{R}^{n} \rightarrow \mathbb{R}$ is a continuous function  and there is a positive constant $R>0$ such that $F \in C_{l, b}^{2}\left(\mathbb{R}^{n}, \mathbb{R}\right)$ when $|x|>R$. For example, one can take $\rho(x)=\exp\{\delta|x|\}$ with $\delta \in \mathbb{R}^+$ or $\rho(x)=(1+|x|)^q$ with $q>n+2$.

\ms

Let $\sH$ be the set of random fields $\left\{u(t, x); 0 \leqslant t \leqslant T, x \in \mathbb{R}^{n}\right\}$ such that $u(t, x)$ is $\mathscr{F}_{t, T}^{B}$-measurable, and both $u$ and $\sigma^\top \nabla u$ belong to $L^{2}\left(\Omega \ts(0, T) \times \mathbb{R}^{n};\mathrm{d} \mathbb{P}\otimes \mathrm{d} t \otimes \rho^{-1}(x) \mathrm{d} x \right)$. Then $\mathscr{H}$ is a Banach space endowed with the following norm:
$$
\|u\|_{\mathscr{H}}^{2}\deq \dbE\left[\int_{\mathbb{R}^{n}} \int_{0}^{T}\left(|u(t, x)|^{2}+\left|(\sigma^\top \nabla u)(t, x)\right|^{2}\right) d t \rho^{-1}(x) d x\right].
$$

Now we present the definition of Sobolev solution of SPDE \rf{SPDE}.
\begin{definition}\label{Sobolev} \rm
We say that $u$ is a Sobolev solution of SPDE \rf{SPDE}, if $u \in \mathscr{H}$ and for any $\varphi \in \mathcal{C}_{c}^{1, \infty}([0, T] \times\mathbb{R}^{n})$,
\begin{equation}\label{Sobolev2}
\begin{aligned}
\ds &\int_{\mathbb{R}^{n}} \int_{t}^{T} u(s, x) \partial_{s} \varphi(s, x) ds dx+\int_{\mathbb{R}^{n}} u(t, x) \varphi(t, x) d x-\int_{\mathbb{R}^{n}} h(x) \varphi(T, x) d x \\
\ns\ds &\quad-\frac{1}{2} \int_{\mathbb{R}^{n}} \int_{t}^{T}(\sigma^\top \nabla u)(s, x) \cdot(\sigma^\top \nabla \varphi)(s, x) d s d x
-\int_{\mathbb{R}^{n}} \int_{t}^{T} u \mathrm{div}[(b-\widetilde{A}) \varphi](s, x) d s d x \\
\ns\ds &=\int_{\mathbb{R}^{n}} \int_{t}^{T} f\big(s, x, u(s, x),(\sigma^\top \nabla u)(s, x)\big) \varphi(s, x) d s dx \\
\ns\ds &\quad+\int_{\mathbb{R}^{n}} \int_{t}^{T} g\big(s, x, u(s, x),(\sigma^\top \nabla u)(s, x)\big) \varphi(s, x)  d B_{s} dx,
\end{aligned}
\end{equation}
where $A$ is a $n$-vector whose coordinates are given by
$\widetilde{A}_j=\frac{1}{2}\sum_{i=1}^n\frac{\partial a_{ij}}{\partial x_i}$ with $1\les j\les n.$
\end{definition}

It should be pointed out that, for the classical solution, indeed, if one supposes that $u$ is a solution of SPDE \rf{SPDE} of class $C^2$, then, similar to the method of Pardoux--Peng \cite{Pardoux-Peng-94}, applying \autoref{Itoformula} to $u(s,X^{t,x}_s)$ implies that the pair of processes $\{Y^{t,x}_s,Z^{t,x}_s;t\les s\les T\}$ defined by
$$
Y^{t,x}_s=u(s, X_{s}^{t, x}) \q~\hb{and}\q~ Z^{t,x}_s=(\sigma^\top \nabla u)(s, X_{s}^{t, x})
$$
is a solution of the backward doubly stochastic differential equation  \rf{FBDSDE}.

\ms

In the following, we discuss this relationship between the Sobolev solution of SPDEs and the solution of BDSDEs with quadratic growth. In particular, we first focus on a simple situation to better reflect the idea of solving the problem.

\subsection{Simple situation}

Consider the following type of backward doubly stochastic differential equations: for $t\les s\les T$,
\begin{equation}\label{BDSDE-B}
Y^{t,x}_s=h(X^{t,x}_T)+\int_s^T\big[f(r,X^{t,x}_r,Y^{t,x}_r,Z^{t,x}_r)+C(Z^{t,x}_r)^2\big]dr
+\int_s^T\alpha Z^{t,x}_rd\oaB_r - \int_s^TZ^{t,x}_{r}dW_{r},
\end{equation}
where  $X^{t,x}$ is the solution of \rf{FSDE},  $C>0$ and $\a\in(-1,1)$ are two constants, and the functions $f:[0,T]\ts\dbR^n\ts\dbR\ts\dbR^d\rightarrow\dbR$ and $h:\dbR^n\rightarrow\dbR$ satisfy the following assumption.
\begin{assumption}\label{A5}\rm
The function $h$ is bounded, and there exist bounded functions $l:[0,T]\rightarrow\dbR^+$ and $k:[0,T]\rightarrow\dbR$ such that for all $t\in[0,T],$ $x\in\dbR^n$, $y\in\dbR$ and $z\in\dbR^d$, the generator $f$ satisfies the following condition:
\begin{equation*}
\begin{aligned}
\ds |f(t,x, y,z)|\les l(t),\q~ \frac{\partial f}{\partial y}(t,x,y,z)
+ \frac{\partial f}{\partial z}(t,x,y,z)&\les k(t),\q~ \text {a.s. }
\end{aligned}
\end{equation*}
\end{assumption}


\begin{proposition}\label{22.3.9.0}
Under \autoref{A5}, the following stochastic partial differential equation
\begin{equation}\label{SPDE0}
\begin{aligned}
\ds u(s, x)=& h(x)+\int_{s}^{T}\left\{
\mathcal{L}u(r, x)
+f\big(r, x, u(r, x),(\sigma^\top \nabla u)(r, x)\big)
+C\big((\sigma^\top \nabla u)(r, x)\big)^2
\right\} dr \\
\ns\ds &+\a\int_{s}^{T}(\sigma^\top \nabla u)(r, x)d\oaB_r,
\quad~ t \leqslant s \leqslant T,
\end{aligned}
\end{equation}
%
admits a unique Sobolev solution $u\in\sH$, and
%
for every $t\in[0,T]$,
\begin{equation*}
u(s, X_{s}^{t, x})= Y^{t,x}_s \q~\hb{and}\q~
(\sigma^\top \nabla u)(s, X_{s}^{t, x})=Z^{t,x}_s, \q~
 a.s., a.e.\ s\in[t,T],\ x\in\dbR^n,
\end{equation*}
where $\{(Y_{s}^{t, x}, Z_{s}^{t, x}); t \les s \les T\}$ is the unique solution of quadratic BDSDE \rf{BDSDE-B}.
\end{proposition}

\begin{proof}
Applying the exponential transformation of variable to $\bar Y=\exp\{\beta Y\}$ transforms
formally BDSDE \rf{BDSDE-B} into the following BDSDE:
\begin{align*}
\ds  \bar Y^{t,x}_s=&\ \bar Y^{t,x}_T+\int_s^T\b\bar Y^{t,x}_rf(r,X^{t,x}_r,Y^{t,x}_r,Z^{t,x}_r)dr
+\Big[C-\frac{1-\a^2}{2}\b\Big]\int_s^T\b \bar Y^{t,x}_r|Z^{t,x}_{r}|^{2}dr\\
\ns\ds&\ +\a\int_s^T\b \bar Y^{t,x}_rZ^{t,x}_{r} d\overleftarrow{B}_r
-\int_s^T\b \bar Y^{t,x}_rZ^{t,x}_rdW_r,\q~ t\les s\les T.
\end{align*}
By taking $\b=\frac{2C}{1-\a^2}$, the above equation becomes
\begin{align}\label{22.3.10.1}
\ds  \bar Y^{t,x}_s= \bar h(X^{t,x}_T)
+\int_s^T\bar f(r,X^{t,x}_r,\bar Y^{t,x}_r,\bar Z^{t,x}_r)dr
 +\a\int_s^T \bar Z^{t,x}_r d\overleftarrow{B}_r
-\int_s^T \bar Z^{t,x}_rdW_r,\q~ t\les s\les T,
\end{align}
where
\begin{align}
\ds \bar Z^{t,x}_r=\b\bar Y^{t,x}Z^{t,x}_r,\q~ \bar h(X^{t,x}_T)=\exp\{\b h(X^{t,x}_T)\}, \nn\\
\ns\ds \bar f(r,X^{t,x}_r,\bar Y^{t,x}_r,\bar Z^{t,x}_r)=\b\bar Y^{t,x}_rf(r,X^{t,x}_r,Y^{t,x}_r,Z^{t,x}_r).\label{22.4.14.5}
\end{align}
Under \autoref{A5}, Eq. \rf{22.3.10.1} is a classical BDSDE with a globally Lipschitz generator. Then Bally--Matoussi \cite{Bally-Matoussi-01} (see also Zhang--Zhao \cite{Zhang-Zhao-10} and Wu--Zhang \cite{Wu-Zhang-11}) implies that the following SPDE
\begin{equation}\label{22.4.14.2}
\begin{aligned}
\ds \bar u(s, x)=& \bar h(x)+\int_{s}^{T}\left\{\mathcal{L}\bar u(r, x)+\bar f\big(r, x, \bar u(r, x),(\sigma^\top \nabla \bar u)(r, x)\big)\right\} dr \\
\ns\ds &+\a\int_{s}^{T}(\sigma^\top \nabla \bar u)(r, x)d\oaB_r,
\quad~ t \leqslant s \leqslant T,
\end{aligned}
\end{equation}
admits a unique Sobolev solution $\bar u\in\sH$, and for every $t\in[0, T]$,
$$
\bar u(s, X_{s}^{t, x})=\bar Y_{s}^{t, x} \q~ \text{and} \q~
\sigma^\top \nabla \bar u(s, X_{s}^{t, x})=\bar Z_{s}^{t, x},\q~
 a.s., a.e.\ s\in[t,T],\ x\in\dbR^n.
$$
On the other hand, we see that
$$\bar Y^{t,x}_s=\exp(\beta Y^{t,x}_s)\q~\hb{and}\q~\bar Z^{t,x}_s=\b \bar Y^{t,x}_sZ^{t,x}_s,\q~ t\les s\les T,$$
where the pair $( Y^{t,x}, Z^{t,x})\in L_{\dbF}^\i(0,T;\dbR)\times L_{\dbF}^2(0,T;\dbR^d)$ is the solution of BDSDE \rf{BDSDE-B}.
So we would like to define
\begin{align}\label{22.4.14.3}
\ds &u(s,X^{t,x}_s)\deq\frac{\ln \bar u(s,X^{t,x}_s)}{\b}=\frac{\ln \bar Y^{t,x}_s}{\b}=Y^{t,x}_s,\q~ t\les s\les T,
\end{align}
and then one can compute that
\begin{align}\label{22.4.14.4}
(\sigma^\top \nabla u)(s,X^{t,x}_s)=\frac{(\sigma^\top \nabla \bar u)(s,X^{t,x}_s)}{\b \bar u(s,X^{t,x}_s)}
=\frac{\bar Z^{t,x}_s}{\b \bar Y^{t,x}_s}=Z^{t,x}_s,\q~ t\les s\les T.
\end{align}
Hence, applying It\^{o}'s formula (see \autoref{Itoformula}) to $u=\ln \bar u/\b$ could transforms formally SPDE \rf{22.4.14.2} into SPDE \rf{SPDE0}. In fact, by using It\^{o}'s formula  to $u=\ln \bar u/\b$, one has that
\begin{equation}\label{22.4.14.6}
\begin{aligned}
  d u(s,x)=&\ \frac{-1}{\b \bar u(s,x)}
  \left\{\mathcal{L}\bar u(s, x)+\bar f\big(s, x, \bar u(s, x),(\sigma^\top \nabla \bar u)(s, x)\big)\right\}ds\\
  & -\frac{\a (\sigma^\top \nabla \bar u)(s, x)}{\b \bar u(s,x)}d\overleftarrow{B}_s
  +\frac{\a^2}{2}\frac{\big( (\sigma^\top \nabla \bar u)(s, x)\big)^2}{\b \bar u(s,x)^2}ds.
\end{aligned}
\end{equation}
From the definitions \rf{Lu} and \rf{22.4.14.3}, it is easy to compute that
\begin{align}\label{22.4.14.7}
  \mathcal{L}\bar u(s, x)=\b \bar u(s,x)\[\mathcal{L} u(s, x)
  +  \frac{\b}{2}\big((\sigma^\top \nabla u)(s, x)\big)^2 \].
\end{align}
Then, combining \rf{22.4.14.5}, \rf{22.4.14.3}, \rf{22.4.14.4}, \rf{22.4.14.6}, \rf{22.4.14.7}, and note that $\b=\frac{2C}{1-\a^2}$, we deduce that
\begin{equation}\label{22.4.15.1}\left\{
\begin{aligned}
\ds du(s, x)=& -\left\{
\mathcal{L}u(s, x)
+f\big(s, x, u(s, x),(\sigma^\top \nabla u)(s, x)\big)
+C\big((\sigma^\top \nabla u)(s, x)\big)^2
\right\} ds \\
\ns\ds &-\a (\sigma^\top \nabla u)(s, x)d\oaB_s,
\quad~ t \leqslant s \leqslant T,\\
\ns\ds u(T,x)=& h(x).
\end{aligned}\right.
\end{equation}
In other words, $u$ is a Sobolev solution of SPDE \rf{SPDE0}, and for every $t\in[0,T]$,
\begin{equation*}
u(s, X_{s}^{t, x})= Y^{t,x}_s \q~ \hb{and}\q~
(\sigma^\top \nabla u)(s, X_{s}^{t, x})=Z^{t,x}_s, \q~
 a.s., a.e.\ s\in[t,T],\ x\in\dbR^n,
\end{equation*}
where $(Y^{t, x}, Z^{t, x})$ is the unique solution of BDSDE \rf{BDSDE-B}.

\ms

Finally, the uniqueness of solutions of SPDE \rf{SPDE0} comes from the uniqueness of BDSDE \rf{BDSDE-B}. This completes the proof.
\end{proof}

\begin{remark}\rm
The idea of the above proof is firstly to transform a simple quadratic BDSDE into a classical BDSDE, and then by using the relationship between the classical BDSDEs and SPDEs, one can conversely obtain the relationship between the BDSDE with quadratic growth and its related SPDE.
%
\end{remark}

\subsection{General situation}

In this subsection, we use BDSDE \rf{FBDSDE} to give a probabilistic representation for the solutions of SPDE \rf{SPDE} in Sobolev space, and use it to prove the existence and uniqueness of Sobolev solutions of SPDE \rf{SPDE}, thus obtaining the nonlinear stochastic Feynman-Kac formula in this framework. The following lemma is an extension of equivalence of norm principle given in Bally--Matoussi \cite{Bally-Matoussi-01} and plays an important role later.

\begin{lemma}\label{22.4.11.5}\rm
There exist two positive constants $k_{1}$ and $K_{1}$ which depend on $T, \rho, b$ and $\sigma$, such that for any $t \leqslant s \leqslant T$ and $\phi \in L^{1}\left(\Omega \times \mathbb{R}^{n}; \mathrm{~d} \mathbb{P} \otimes \rho^{-1}(x) \mathrm{d} x\right)$ which is independent of $\mathscr{F}_{t, s}^{W}$,
$$
k_{1} \dbE\int_{\mathbb{R}^{d}}|\phi(x)| \rho^{-1}(x) \mathrm{d} x
\leqslant \dbE\int_{\mathbb{R}^{d}}\left|\phi\left(X_{s}^{t, x}\right)\right| \rho^{-1}(x) \mathrm{d} x \leqslant K_{1} \dbE\int_{\mathbb{R}^{d}}|\phi(x)| \rho^{-1}(x) \mathrm{d} x.
$$
Moreover, for any $\Phi \in L^{1}\left(\Omega \times[0, T] \times \mathbb{R}^{d}; \mathrm{~d} \mathbb{P} \otimes \mathrm{d} t \otimes \rho^{-1}(x) \mathrm{d} x\right)$ such that $\Phi(s, \cdot)$ is independent of $\mathscr{F}_{t, s}^{W}$, one has
$$
\begin{aligned}
k_{1} \dbE\int_{\mathbb{R}^{d}} \int_{t}^{T}|\Phi(s, x)| \mathrm{d} s \rho^{-1}(x) \mathrm{d} x & \leqslant \dbE\int_{\mathbb{R}^{d}} \int_{t}^{T}\left|\Phi\left(s, X_{s}^{t, x}\right)\right| \mathrm{d} s \rho^{-1}(x) \mathrm{d} x \\
& \leqslant K_{1} \dbE\int_{\mathbb{R}^{d}}\int_{t}^{T}|\Phi(s, x)|\mathrm{d}s \rho^{-1}(x)\mathrm{d}x.
\end{aligned}
$$
\end{lemma}

For the coefficients of BDSDE \rf{FBDSDE}, we present the following assumptions.

\begin{assumption}\label{A6}\rm
For all $t\in[0,T],$ $x\in\dbR^n$, $y\in\dbR$ and $z\in\dbR^d$,
the function $h:\dbR^n\rightarrow\dbR$ is bounded, and the functions
 $f:[0,T]\ts\dbR^n\ts\dbR\ts\dbR^d\rightarrow\dbR$ and
 $g:[0,T]\ts\dbR^n\ts\dbR\ts\dbR^d\rightarrow\dbR^l$ satisfy \autoref{H3}.
\end{assumption}

\begin{remark}\rm
Note that under \autoref{A6},  BDSDE \rf{FBDSDE} has a unique solution  $( Y^{t,x}, Z^{t,x})\in L_{\dbF}^\i(0,T;\dbR)\times L_{\dbF}^2(0,T;\dbR^d)$. Moreover, the coefficients $f$, $g$ and $h$ satisfy that
\begin{equation*}
\int_{\dbR^n}\int_0^T\(|h(x)|^2+|f(t,x,0,0)|^2+|g(t,x,0,0)|^2\)dt\rho^{-1}(x)dx<\i.
\end{equation*}
\end{remark}

Now we are in a position to present the main result of this section.

\begin{theorem}[Feynman-Kac formula]\label{22.3.9.1}
Under \autoref{A6}, SPDE \rf{SPDE} admits a unique Sobolev solution $u\in\sH$, and for every $t\in[0,T]$,
\begin{equation*}\label{22.4.13.5}
u(s, X_{s}^{t, x})= Y^{t,x}_s \q~\hb{and}\q~
(\sigma^\top \nabla u)(s, X_{s}^{t, x})=Z^{t,x}_s, \q~
 a.s., a.e.\ s\in[t,T],\ x\in\dbR^n,
\end{equation*}
where $\{(Y_{s}^{t, x}, Z_{s}^{t, x}); t \les s \les T\}$ is the unique solution of BDSDE \rf{FBDSDE} with quadratic growth.
\end{theorem}

\begin{proof}

Uniqueness. The uniqueness of SPDE \rf{SPDE} follows from the uniqueness of BDSDE
\rf{FBDSDE}. In fact, if $u^1$ and $u^2$ are two Sobolev solutions of SPDE \rf{SPDE}, then the following two pairs
\begin{equation*}
\big(u^1(s, X_{s}^{t, x}), (\sigma^\top \nabla u^1)(s, X_{s}^{t, x})\big)\q~\hb{and}\q~
\big(u^2(s, X_{s}^{t, x}), (\sigma^\top \nabla u^2)(s, X_{s}^{t, x})\big)
\end{equation*}
solve BDSDE \rf{FBDSDE}. So the uniqueness of BDSDE \rf{FBDSDE} gives us that
\begin{equation*}
u^1(s, X_{s}^{t, x})=u^2(s, X_{s}^{t, x}), \q a.s., a.e.\ s\in[t,T],\ x\in\dbR^n,
\end{equation*}
and in particular
\begin{equation*}
u^1(t,x)=u^2(t,x), \q a.s., a.e.,\ x\in\dbR^n.
\end{equation*}

Next, we prove the existence. Compared with the proof of the simple situation (see \autoref{22.3.9.0}), the main idea here we used is similar, however, the techniques are more complex and include an approximation. So the proof will be divided into six steps with some techniques borrowed from the proof of \autoref{existence}.

\ms

$\emph{Step 1. The exponential transformation.}$

\ms

Let $\b=\frac{2C}{1-\a}$ be a given positive constant.
Applying the exponential change of variable (see \autoref{Itoformula}) to $\bar Y^{t,x}=\exp(\beta Y^{t,x})$ could transforms formally BDSDE \rf{FBDSDE} with parameters $(f,g,h)$ into the following BDSDE with parameters $(\bar f, \bar g,\bar h)$,
\begin{align}\label{3.9.2.2}
 \bar Y^{t,x}_t=\bar h(X^{t,x}_T)+\int_t^T\bar f(s,X^{t,x}_s,\bar Y^{t,x}_s,\bar Z^{t,x}_s)ds+\int_t^T\bar g(s,X^{t,x}_s,\bar Y^{t,x}_s,\bar Z^{t,x}_s)d\oaB_s - \int_t^T\bar Z^{t,x}_{s}dW_{s},
\end{align}
where $\bar Z^{t,x}_s=\b \bar Y^{t,x}_sZ^{t,x}_s$ and
\begin{equation*}
\begin{aligned}
\ds\bar f(s,X^{t,x}_s,\bar Y^{t,x}_s,\bar Z^{t,x}_s)=
&\ \b \bar Y^{t,x}_s f\(s,X^{t,x}_s,\frac{\ln \bar Y^{t,x}_s}{\b},\frac{\bar Z^{t,x}_s}{\b \bar Y^{t,x}_s}\)\\
\ns\ds&\ +\frac{1}{2}\b^2\bar Y^{t,x}_s\Big[g\(s,X^{t,x}_s,\frac{\ln \bar Y^{t,x}_s}{\b},\frac{\bar Z^{t,x}_s}{\b \bar Y^{t,x}_s}\)^2
  -\Big(\frac{\bar Z^{t,x}_s}{\b \bar Y^{t,x}_s}\Big)^2\Big],\\
\ds  \bar g(s,X^{t,x}_s,\bar Y^{t,x}_s,\bar Z^{t,x}_s)
  =&\ \b \bar Y^{t,x}_s g\(s,X^{t,x}_s,\frac{\ln \bar Y^{t,x}_s}{\b},\frac{\bar Z^{t,x}_s}{\b \bar Y^{t,x}_s}\),\\
\ns\ds  \bar h(X^{t,x}_T)=&\ \exp(\b h(X^{t,x}_T)),\q~ t\les s\les T.
\end{aligned}
\end{equation*}
Then the pair of processes $(\bar Y^{t,x},\bar Z^{t,x})$ is the solution of BDSDE \rf{3.9.2.2} with parameters $(\bar f, \bar g,\bar h)$, where $\bar h$ is bounded and $\bar g$ satisfies \autoref{H2}.

\ms

$\emph{Step 2. Approximation of SPDE.}$

\ms

Let the sequence $\{\bar{f}^n;n\in\dbN\}$ be defined as in \rf{22.4.9.1}. Then, the same arguments we developed in the proof of \autoref{existence} deduce that $\bar f^{n}$ is globally Lipschitz continuous with the constant $n$ and decreasing converges pointwise to the function $\bar f$.
Now, for each $(t,x)\in[0,T]\ts\dbR^n$ and $n\in\dbN$, we define the processes $\bar u^{n}$ and $\bar v^{n}$ by
\begin{equation}\label{22.4.11.1}
  \bar u^{n}(t,x)=\bar{Y}^{t,x,n}_t\q~\hb{and}\q~ \bar v^{n}(t,x)=\bar Z^{t,x,n}_t,
\end{equation}
where $(\bar Y^{t,x,n},\bar Z^{t,x,n})$ is the unique solution of the following BDSDE with parameters $(\bar f^n,\bar g,\bar h)$: for $t\in[0,T]$ and $x\in\dbR^n$,
\begin{equation}\label{22.4.11.2}
\begin{aligned}
\ds \bar Y^{t,x,n}_t=&\
\bar h(X^{t,x}_T)+\int_t^T\bar f^n(s,X^{t,x}_s,\bar Y^{t,x,n}_s,\bar Z^{t,x,n}_s)ds\\
\ns\ds & +\int_t^T\bar g(s,X^{t,x}_s,\bar Y^{t,x,n}_s,\bar Z^{t,x,n}_s)d\oaB_s - \int_t^T\bar Z^{t,x,n}_{s}dW_{s}.
\end{aligned}
\end{equation}
Note that the parameters $(\bar f^n,\bar g,\bar h)$ of the above BDSDE \rf{22.4.11.2} satisfy that
\begin{equation*}
\int_{\dbR^n}\int_0^T\(|\bar h(x)|^2+|\bar f^n(t,x,\bar y,0)|^2+|\bar g(t,x,\bar y,0)|^2\)dt\rho^{-1}(x)dx<\i,\q
\forall \bar y\in\big[\exp\{-\b M\},\exp\{\b M\}\big].
\end{equation*}
Then Bally--Matoussi \cite{Bally-Matoussi-01} (see also Zhang--Zhao \cite{Zhang-Zhao-10} and Wu--Zhang \cite{Wu-Zhang-11}) deduce that
\begin{equation}\label{22.4.19.1}
\bar v^{n}(s,x)=(\sigma^\top\nabla \bar u^n)(s,x),
\end{equation}
and $\bar u^n$ is the unique Sobolev solution to the following SPDE:
\begin{equation}\label{SPDE.2}
\begin{aligned}
\ds \bar u^n(s, x)=& \bar h(x)+\int_{s}^{T}\left\{\mathcal{L} \bar u^n(r, x)
+\bar f^n\big(r, x, \bar u^n(r, x),(\sigma^\top \nabla \bar u^n)(r, x)\big)\right\} dr \\
\ns\ds &+\int_{s}^{T} \bar g\big(r, x, \bar u^n(r, x),(\sigma^\top \nabla \bar u^n)(r, x)\big)d\oaB_r, \quad s\in[t,T],
\end{aligned}
\end{equation}
with
\begin{equation}\label{22.4.12.1}
\bar u^{n}\left(s, X_{s}^{t, x}\right)=\bar Y_{s}^{t, x, n} \quad~\hb{and}\q~\big(\sigma^\top \nabla \bar u^{n}\big)\left(s, X_{s}^{t, x}\right)=\bar Z_{s}^{t, x, n}, \q~ a.s., a.e.\ s\in[t,T],\ x\in\dbR^n,
\end{equation}
where $(\bar Y^{t,x,n},\bar Z^{t,x,n})$ is the unique solution of BDSDE \rf{22.4.11.2} with parameters $(\bar f^n,\bar g,\bar h)$.

\ms

On the other hand, similar to \rf{22.4.11.3}, the comparison theorem for Lipschitz continuous coefficients (see \autoref{Shi05}) implies that the sequence $\{\bar Y^{t, x, n};n\in\dbN\}$ is bounded and decreasing in $n$ such that
\begin{equation}\label{22.4.11.4}
\exp(-\b \|h\|_\i) \les \bar Y_s^{t, x, n+1}\les \bar Y_s^{t, x, n}\les \exp(\b \|h \|_{\i}),\q~ \hb{a.s.}\ s\in[t,T].
\end{equation}
In addition, by using the same arguments we developed in the proofs of \autoref{PriEstimate} and \autoref{Mono-stable}, there exists a positive constant $K$ such that
\begin{equation}\label{22.4.1.6}
 \sup_{n\in\dbN} \dbE\int_t^T |\bar Z^{t, x, n}_s|^2ds\les K,
\end{equation}
and
\begin{equation}\label{22.4.11.7}
\left\{\begin{aligned}
\ds &\lim_{n\rightarrow\i}\bar Y^{t, x, n}_s=\bar Y^{t, x}_s,\q~ \hb{a.s.}\ s\in[t,T],\\
 \ns\ds  & \lim_{n\rightarrow\i} \dbE\int_t^T |\bar Z^{t, x, n}_s-\bar Z^{t, x}_s|^2ds=0,
\end{aligned}\right.
\end{equation}
where $(\bar Y^{t, x},\bar Z^{t, x})\in L_{\dbF}^\i(0,T;\dbR)\times L_{\dbF}^2(0,T;\dbR^d)$ is the solution of BDSDE \rf{3.9.2.2} with parameters $(\bar f,\bar g,\bar h)$.

\ms

$\emph{Step 3. Convergence of SPDE \rf{SPDE.2}}$

\ms

We point out that the limits which we consider below hold along with a subsequence, however for simplicity, the subsequence will still be indexed by $n$. From \autoref{22.4.11.5}, note \rf{22.4.19.1}, \rf{22.4.12.1}, \rf{22.4.11.4} and \rf{22.4.1.6}, we have that
\begin{equation}\label{22.4.19.2}
\begin{aligned}
\ds & \dbE \int_{\mathbb{R}^{n}} \int_{t}^{T}\left(\left|\bar u^{n}(s, x)\right|^{2}
+\left|\bar v^{n}(s, x)\right|^{2}\right) d s \rho^{-1}(x) d x \\
\ns\ds & \les K_{1} \dbE\int_{\mathbb{R}^{n}}\(
T\|\bar Y^{t,x,n}\|_{L^\i_{\dbF}(t,T)}^2+\int_t^T|\bar Z^{t,x,n}_s|^2ds\)\rho^{-1}(x) d x \\
\ns\ds & \les K_1 \int_{\mathbb{R}^{n}}\big[T\exp(2\b \|h\|_\i)+K\big] \rho^{-1}(x) d x
<\infty.
\end{aligned}
\end{equation}
Moreover, similar to \rf{22.4.11.6}, we have
$$
|\bar g\big(s, x, \bar u^n(s, x),\bar v^n(s, x)\big)|^2\les \a|\bar v^n(s,x)|^2,
 \q~ a.s., a.e.\ s\in[t,T],\ x\in\dbR^n.
$$
Then, note that $\bar f^{n}$ is globally Lipschitz continuous, we deduce from \rf{22.4.19.2} that
\begin{equation*}
  \dbE\int_{\dbR^n}\int_t^T\[|\bar f^n\big(s, x, \bar u^n(s, x),\bar v^n(s, x)\big)|^2
  +|\bar g\big(s, x, \bar u^n(s, x),\bar v^n(s, x)\big)|^2\]ds\rho^{-1}(x)dx\les \widetilde{K},
\end{equation*}
where $\widetilde{K}$ is some constant which depends on $\|h\|_\i$, $K$, $K_1$ and $T$.
Now, from \autoref{22.4.11.5}, \rf{22.4.11.1}, \rf{22.4.11.7} and Lebesgue dominated convergence theorem, we have that
\begin{align*}
\ds &\lim_{n,m}\int_{\dbR^n}\int_t^T|\bar u^n(s,x)-\bar u^m(s,x)|^2ds\rho^{-1}(x)dx=0,\\
\ns\ds &\lim_{n,m}\int_{\dbR^n}\int_t^T|(\sigma^\top \nabla \bar u^n)(s,x)
-(\sigma^\top \nabla \bar u^m)(s,x)|^2ds\rho^{-1}(x)dx=0.
\end{align*}
Then, using \autoref{22.4.11.5} again, and note that the space $\sH$ is complete, it follows that there exists $\bar u\in\sH$ such that
\begin{align}
\ds &\lim_{n}\int_{\dbR^n}\int_t^T|\bar u^n(s,x)-\bar u(s,x)|^2ds\rho^{-1}(x)dx=0, \label{22.4.12.3}\\
\ns\ds &\lim_{n}\int_{\dbR^n}\int_t^T|(\sigma^\top \nabla \bar u^n)(s,x)
-(\sigma^\top \nabla \bar u)(s,x)|^2ds\rho^{-1}(x)dx=0, \label{22.4.12.4}\\
\ns\ds &\int_{\dbR^n}\int_t^T\(|\bar u(s,x)|^2+(\sigma^\top \nabla \bar u)(s,x)|^2\)ds\rho^{-1}(x)dx<\i.\nn
\end{align}

\ms

\emph{Step 4. We show that for any $t\in[0,T]$,
\begin{equation}\label{22.4.12.7}
\bar u(s,X^{t,x}_s)=\bar Y^{t,x}_s\q \hb{and}\q~
(\sigma^\top \nabla \bar u)(s,X^{t,x}_s)=\bar Z^{t,x}_s,
 \q~ a.s., a.e.\ s\in[t,T],\ x\in\dbR^n.
\end{equation}}
By triangular inequality, it is easy to see that
\begin{align}
\ds &\dbE\int_{\dbR^n}\int_t^T|\bar u(s,X^{t,x}_s)-\bar Y^{t,x}_s|^2ds\rho^{-1}(x)dx\nn\\
\ns\ds &\les  \dbE\int_{\dbR^n}\int_t^T|\bar u(s,X^{t,x}_s)-\bar u^n(s,X^{t,x}_s)|^2ds\rho^{-1}(x)dx
\label{22.4.12.2}\\
\ns\ds&\q + \dbE\int_{\dbR^n}\int_t^T|\bar u^n(s,X^{t,x}_s)-\bar Y^{t,x}_s|^2ds\rho^{-1}(x)dx.  \label{22.4.12.8}
\end{align}
From \autoref{22.4.11.5} and \rf{22.4.12.3}, one has that the term \rf{22.4.12.2} tends to $0$ as $n$ tends to $\i$. In addition, \autoref{22.4.11.5}, \rf{22.4.12.1} and \rf{22.4.11.7} deduce that the term \rf{22.4.12.8} tends to $0$ as $n$ tends to $\i$.
So the first result in \rf{22.4.12.7} holds. We next prove the second result in \rf{22.4.12.7}, i.e.,
\begin{equation}\label{22.4.12.5}
(\sigma^\top \nabla \bar u)(s,X^{t,x}_s)=\bar Z^{t,x}_s,  \q~ a.s., a.e.\ s\in[t,T],\ x\in\dbR^n.
\end{equation}
It is sufficient to show that the following term
\begin{equation}\label{22.4.12.6}
\begin{aligned}
\ds &\dbE\int_{\dbR^n}\int_t^T| (\sigma^\top \nabla \bar u)(s,X^{t,x}_s)-\bar Z^{t,x}_s|^2ds\rho^{-1}(x)dx  \\
\ns\ds &\les  \dbE\int_{\dbR^n}\int_t^T| (\sigma^\top \nabla \bar u)(s,X^{t,x}_s)
- (\sigma^\top \nabla \bar u^n)(s,X^{t,x}_s)|^2ds\rho^{-1}(x)dx\\
\ns\ds&\q~ + \dbE\int_{\dbR^n}\int_t^T| (\sigma^\top \nabla \bar u^n)(s,X^{t,x}_s)-\bar Z^{t,x}_s|^2ds\rho^{-1}(x)dx\\
\ns\ds & \longrightarrow 0 \q \hb{as}\q n\rightarrow \i.
\end{aligned}
\end{equation}
In fact, similarly, combining \autoref{22.4.11.5}, \rf{22.4.12.1}, \rf{22.4.11.7}, and \rf{22.4.12.4}, it is easy to see that \rf{22.4.12.6} holds, which implies that \rf{22.4.12.5} holds too.

\ms

\emph{Step 5. We show that $\bar u$ is a Sobolev solution to the following SPDE:}
\begin{equation}\label{SPDE.3}
\begin{aligned}
\ds \bar u(s, x)=& \bar h(x)+\int_{s}^{T}\left\{\mathcal{L} \bar u(r, x)
+\bar f\big(r, x, \bar u(r, x),(\sigma^\top \nabla \bar u)(r, x)\big)\right\} dr \\
\ns\ds &+\int_{s}^{T} \bar g\big(r, x, \bar u(r, x),(\sigma^\top \nabla \bar u)(r, x)\big)d\oaB_r, \quad s\in[t,T].
\end{aligned}
\end{equation}
First, it was shown in {\it Step 3} that $\bar u$ belongs to the space $\sH$, so by \autoref{Sobolev}, we only need to verify that $\bar u$ satisfies Eq. \rf{Sobolev2} with parameters $(\bar f,\bar g,\bar h)$.
Note that for every $n\in\dbN$, since $\bar u^n$ is the Sobolev solution of SPDE \rf{SPDE.2}, so for any
 $\varphi \in \mathcal{C}_{c}^{1, \infty}([0, T] \times\mathbb{R}^{n})$, $\bar u^n$ satisfies
\begin{equation}\label{22.4.12.9}
\begin{aligned}
\ds &\int_{\mathbb{R}^{n}} \int_{t}^{T} \bar u^n(s, x) \partial_{s} \varphi(s, x) ds dx+\int_{\mathbb{R}^{n}} \bar u^n(t, x) \varphi(t, x) d x-\int_{\mathbb{R}^{n}} \bar h(x) \varphi(T, x) d x \\
\ns\ds &\quad-\frac{1}{2} \int_{\mathbb{R}^{n}} \int_{t}^{T}(\sigma^\top \nabla \bar u^n)(s, x) \cdot(\sigma^\top \nabla \varphi)(s, x) d s d x
-\int_{\mathbb{R}^{n}} \int_{t}^{T} \bar u^n \mathrm{div}[(b-\widetilde{A}) \varphi](s, x) d s d x \\
\ns\ds &=\int_{\mathbb{R}^{n}} \int_{t}^{T} \bar f^n\big(s, x, \bar u^n(s, x),(\sigma^\top \nabla \bar u^n)(s, x)\big) \varphi(s, x) d s dx \\
\ns\ds &\quad+\int_{\mathbb{R}^{n}} \int_{t}^{T} \bar g\big(s, x, \bar u^n(s, x),(\sigma^\top \nabla \bar u^n)(s, x)\big) \varphi(s, x)  d B_{s} dx.
\end{aligned}
\end{equation}
Hence, if we can prove that along a subsequence \rf{22.4.12.9} converges to \rf{Sobolev2} with parameters $(\bar f,\bar g,\bar h)$, then $\bar u$ satisfies  \rf{Sobolev2} with parameters $(\bar f,\bar g,\bar h)$.

\ms

Due to that $\varphi \in \mathcal{C}_{c}^{1, \infty}([0, T] \times\mathbb{R}^{n})$,
$b\in C_{l, b}^{2}(\dbR^n;\dbR^n)$ and $\sigma\in C_{l, b}^{3}(\dbR^n;\dbR^{n\times n})$, then as $n\rightarrow \i$, clearly the left hand side of \rf{22.4.12.9} tends to the following term:
\begin{equation*}
\begin{aligned}
\ds &\int_{\mathbb{R}^{n}} \int_{t}^{T} \bar u(s, x) \partial_{s} \varphi(s, x) ds dx+\int_{\mathbb{R}^{n}} \bar u(t, x) \varphi(t, x) d x-\int_{\mathbb{R}^{n}} \bar h(x) \varphi(T, x) d x \\
\ns\ds &\quad-\frac{1}{2} \int_{\mathbb{R}^{n}} \int_{t}^{T}(\sigma^\top \nabla \bar u)(s, x) \cdot(\sigma^\top \nabla \varphi)(s, x) d s d x
-\int_{\mathbb{R}^{n}} \int_{t}^{T} \bar u \mathrm{div}[(b-\widetilde{A}) \varphi](s, x) d s d x.
\end{aligned}
\end{equation*}
Next, we compute the limit on the right hand side of \rf{22.4.12.9}.
We show that along a sequence
\begin{equation}\label{22.4.13.3}
\begin{aligned}
\ds & \lim_{n\rightarrow\i}\int_{\mathbb{R}^{n}} \int_{t}^{T} \bar f^n\big(s, x, \bar u^n(s, x),(\sigma^\top \nabla \bar u^n)(s, x)\big) \varphi(s, x) d s dx\\
\ns\ds &=\int_{\mathbb{R}^{n}} \int_{t}^{T} \bar f\big(s, x, \bar u(s, x),(\sigma^\top \nabla \bar u)(s, x)\big) \varphi(s, x) d s dx.
\end{aligned}
\end{equation}
In fact, note that the sequence $\{\bar{f}^n;n\in\dbN\}$ is globally Lipschitz continuous and decreasing converges pointwise to the function $\bar f$, and for every $t\in[0,T]$,
\begin{align*}
& \big(\bar u^{n}(s, X_{s}^{t, x}),(\sigma^\top \nabla \bar u^{n})(s, X_{s}^{t, x})\big)
=\big(\bar Y_{s}^{t, x, n},\bar Z_{s}^{t, x, n}\big), \q~ a.s., a.e.\ s\in[t,T],\ x\in\dbR^n,\\
& \big(\bar u(s, X_{s}^{t, x}),(\sigma^\top \nabla \bar u)(s, X_{s}^{t, x})\big)
=\big(\bar Y_{s}^{t, x},\bar Z_{s}^{t, x}\big),  \q~ a.s., a.e.\ s\in[t,T],\ x\in\dbR^n,
\end{align*}
with
\begin{equation*}
\left\{\begin{aligned}
\ds &\lim_{n\rightarrow\i}\bar Y^{t, x, n}_s=\bar Y^{t, x}_s,\q~ \hb{a.s.}\ s\in[t,T],\\
 \ns\ds  & \lim_{n\rightarrow\i} \dbE\int_t^T |\bar Z^{t, x, n}_s-\bar Z^{t, x}_s|^2ds=0.
\end{aligned}\right.
\end{equation*}
Then, similar to the argument as in the proof of \autoref{Mono-stable}, we get that
\begin{align*}
\ds \lim_{n\rightarrow\i}\dbE\int_t^T
\Big|& \bar f^n\big(s,  X_{s}^{t, x}, \bar u^n(s,  X_{s}^{t, x}),(\sigma^\top \nabla \bar u^n)(s, X_{s}^{t, x})\big)\\
\ns\ds & -\bar f\big(s,X_{s}^{t, x},\bar u(s,X_{s}^{t, x}),(\sigma^\top\nabla\bar u)(s,X_{s}^{t, x})\big)\Big|ds=0.
\end{align*}
So combining \autoref{22.4.11.5} and the Lebesgue dominated convergence theorem we have that
\begin{align*}
\ds \lim_{n\rightarrow\i}\dbE\int_{\dbR^n}\int_t^T
\Big|& \bar f^n\big(s, x, \bar u^n(s,  x),(\sigma^\top \nabla \bar u^n)(s,x)\big)\\
\ns\ds & -\bar f\big(s,x,\bar u(s,x),(\sigma^\top\nabla\bar u)(s,x)\big)\Big|ds\rho^{-1}(x)dx=0,
\end{align*}
which implies that \rf{22.4.13.3} holds. Finally, for the second term on the right hand side of \rf{22.4.12.9}, similar to \rf{22.4.13.4} and the previous discussion, one has that
\begin{align*}
\ds \lim_{n\rightarrow\i}
\bigg|& \int_{\mathbb{R}^{n}} \int_{t}^{T} \bar g\big(s, x, \bar u^n(s, x),(\sigma^\top \nabla \bar u^n)(s, x)\big) \varphi(s, x)  d B_{s} dx\\
\ns\ds & -\int_{\mathbb{R}^{n}} \int_{t}^{T} \bar g\big(s, x, \bar u(s, x),(\sigma^\top \nabla \bar u)(s, x)\big) \varphi(s, x)  d B_{s} dx\bigg|=0\q\hb{in probability.}
\end{align*}
Therefore $\bar u$ satisfies Eq. \rf{Sobolev2} with parameters $(\bar f,\bar g,\bar h)$, and thus $\bar u$ is a Sobolev solution to SPDE \rf{SPDE.3}.

\ms

\emph{Step 6. The logarithmic transformation.}

\ms

On the one hand, we see that SPDE \rf{SPDE.3} with parameters $(\bar f,\bar g,\bar h)$ has a Sobolev solution $\bar u$, and for every $t\in[0,T]$,
\begin{equation*}
\bar u(s,X^{t,x}_s)=\bar Y^{t,x}_s\q \hb{and}\q
(\sigma^\top \nabla \bar u)(s,X^{t,x}_s)=\bar Z^{t,x}_s, \q~ a.s., a.e.\ s\in[t,T],\ x\in\dbR^n,
\end{equation*}
where the pair $(\bar Y^{t,x},\bar Z^{t,x})$ is the solution of BDSDE \rf{3.9.2.2} with parameters $(\bar f, \bar g,\bar h)$. On the other hand,
$$\bar Y^{t,x}_s=\exp(\beta Y^{t,x}_s)\q~\hb{and}\q~\bar Z^{t,x}_s=\b \bar Y^{t,x}_sZ^{t,x}_s,\q~ t\les s\les T,$$
where the pair $( Y^{t,x}, Z^{t,x})\in L_{\dbF}^\i(0,T;\dbR)\times L_{\dbF}^2(0,T;\dbR^d)$ is the solution of BDSDE \rf{FBDSDE} with parameters $( f, g, h)$.
Hence, similar to the computation from \rf{22.4.14.3} to \rf{22.4.15.1}, if we define
\begin{align*}
\ds &u(s,X^{t,x}_s)\deq\frac{\ln \bar u(s,X^{t,x}_s)}{\b}=\frac{\ln \bar Y^{t,x}_s}{\b}=Y^{t,x}_s,\q~ t\les s\les T,
\end{align*}
with
\begin{align*}
(\sigma^\top \nabla u)(s,X^{t,x}_s)=\frac{(\sigma^\top \nabla \bar u)(s,X^{t,x}_s)}{\b \bar u(s,X^{t,x}_s)}
=\frac{\bar Z^{t,x}_s}{\b \bar Y^{t,x}_s}=Z^{t,x}_s,\q~ t\les s\les T,
\end{align*}
then $u$ is a Sobolev solution of SPDE \rf{SPDE} with parameters $( f, g, h)$, and for every $t\in[0,T]$,
\begin{equation*}
u(s, X_{s}^{t, x})= Y^{t,x}_s \q~ \hb{and}\q~
(\sigma^\top \nabla u)(s, X_{s}^{t, x})=Z^{t,x}_s, \q~ a.s., a.e.\ s\in[t,T],\ x\in\dbR^n,
\end{equation*}
where $(Y^{t, x}, Z^{t, x})$ is the unique solution of BDSDE \rf{FBDSDE} with parameters $( f, g, h)$. This completes the proof.
\end{proof}

\section{Conclusion Remarks}\label{Sec6}

We initiate the study of quadratic BDSDEs, where the existence, comparison, and stability results for one-dimensional BDSDEs with quadratic growth and bounded terminal value are obtained.
We remark that when proving a priori estimate, in order to ensure $Y$ is bounded, we introduce the condition \rf{22.1.16.1} concerning the coefficient $g$, which we think is necessary for the case of bounded terminal value. However, if someone discusses the solution of quadratic BDSDEs with unbounded terminal value, the condition \rf{22.1.16.1}  could be relaxed because  the boundedness of $Y$ is needn't at that time. We hope to publish some follow-up works when the terminal value $\xi$ is unbounded or/and $Y$ is multi-dimensional.
Finally, in this framework, we used BDSDEs to prove the existence and uniqueness of the Sobolev solutions of semilinear SPDEs, which extend the nonlinear Feynman-Kac formula of Pardoux--Peng \cite{Pardoux-Peng-94}. On the other hand, the existence and uniqueness of viscosity solutions of related SPDEs are interesting and merit further studies in the near future.

\end{document}